\documentclass{amsart}

\usepackage[dvipsnames,svgnames,x11names,hyperref]{xcolor}
\usepackage{url}
\usepackage[pagebackref,colorlinks,citecolor=Mahogany,linkcolor=Mahogany,urlcolor=Mahogany,filecolor=Mahogany]{hyperref}
\usepackage{lmodern}
\usepackage{microtype}
\usepackage{stmaryrd}
\usepackage{silence}
\usepackage{upgreek}
\usepackage{faktor}
\usepackage{verbatim}
\usepackage{amsmath} 
\usepackage{amssymb}
\usepackage{mathrsfs}
\usepackage{tikz}
\usepackage{tikz-cd}
\usepackage{adjustbox}
\usepackage{mathtools}
\usepackage[mathcal]{euscript}
\usepackage[T1]{fontenc}
\usepackage{enumerate}
\usepackage{indentfirst}
\usepackage{subcaption}

\usepackage{ragged2e}

\newcommand{\Hom}{\mr{Hom}}

\newcommand{\Conf}{\mr{Conf}}

\newcommand{\hofibre}{\mr{hofibre}}

\newcommand{\Emb}{\mr{Emb}}
\newcommand{\Diff}{\mr{Diff}}

\newcommand{\Top}{\mr{Top}}
\newcommand{\TTop}{\mathcal{T}\mr{op}}

\newcommand{\id}{\mr{id}}

\newcommand{\D}{\mathcal{D}}

\newcommand{\A}{\mathcal{A}}
\newcommand{\B}{\mathcal{B}}

\newcommand{\C}{\mathcal{C}}

\newcommand{\Map}{\mr{Map}}

\DeclareMathOperator*{\holim}{\mr{holim}}
\DeclareMathOperator*{\hocolim}{\mr{hocolim}}

\DeclareMathOperator*{\colim}{\mr{colim}}

\newcommand{\Retr}{\mr{Retr}}
\newcommand{\Mic}{\mr{Mic}}
\newcommand{\MMic}{\mathcal{M}\mr{ic}}
\newcommand{\MMfld}{\mathcal{M}\mr{fld}}
\newcommand{\DDisk}{\mathcal{D}\mr{isk}}
\newcommand{\PSh}{\mathcal{P}\mr{sh}}

\newcommand{\MMet}{\mathcal{M}\mr{et}}

\newcommand{\VVec}{\mathcal{V}\mr{ec}}

\newcommand{\tp}{t}
\newcommand{\s}{s}
\newcommand{\f}{f}
\renewcommand{\r}{r}

\theoremstyle{definition}\newtheorem{definition}{Definition}[section]
\theoremstyle{remark}\newtheorem{remark}[definition]{Remark}
\theoremstyle{definition}\newtheorem{notation}[definition]{Notation}
\theoremstyle{definition}
\theoremstyle{remark}\newtheorem{example}[definition]{Example}
\theoremstyle{plain}\newtheorem{proposition}[definition]{Proposition}
\theoremstyle{plain}\newtheorem{lemma}[definition]{Lemma}
\theoremstyle{definition}
\theoremstyle{remark}\newtheorem{question}[definition]{Question}
\theoremstyle{plain}\newtheorem{corollary}[definition]{Corollary}
\theoremstyle{plain}\newtheorem{theorem}[definition]{Theorem}
\theoremstyle{definition}\newtheorem{construction}[definition]{Construction}
\theoremstyle{remark}
\theoremstyle{plain}
\theoremstyle{plain}
\newtheorem{atheorem}{Theorem}

\newcommand{\lra}{\longrightarrow}

\newcommand{\mr}[1]{{\rm #1}}

\makeatletter
\@addtoreset{definition}{section}
\makeatother

\AtBeginDocument{%
	\def\MR#1{}
}

\title{Embedding calculus and smooth structures}
\author{Ben Knudsen}
\address{Department of Mathematics, Northeastern University, Boston, MA 02115, USA}
\email{b.knudsen@northeastern.edu}
\author{Alexander Kupers}
\address{University of Toronto,	Department of Mathematics, Toronto, ON, M1C 1A4 Canada}
\email{a.kupers@utoronto.ca}
\date{} 

\begin{document}

\begin{abstract} We study the dependence of the embedding calculus Taylor tower on the smooth structures of the source and target. We prove that embedding calculus does not distinguish exotic smooth structures in dimension 4, implying a negative answer to a question of Viro. In contrast, we show that embedding calculus does distinguish certain exotic spheres in higher dimensions. As a technical tool of independent interest, we prove an isotopy extension theorem for the limit of the embedding calculus tower, which we use to investigate several further examples.
\end{abstract}

\maketitle

\section{Introduction}

This paper is an investigation into the scope of a certain tool used to study the space $\Emb^\s(N,M)$ of smooth embeddings from an $n$-manifold $N$ into an $m$-manifold $M$. This investigation has consequences for spaces of embeddings themselves, as shown by the following result on knots and links, which answers a question of Viro \cite[\S 6]{Viro:SS1} in the negative and improves on a result of Arone and Szymik \cite{AroneSzymik:SKCESS}.

\begin{atheorem}\label{athm:knots}
Let $M$ and $M'$ be smooth, simply connected, compact $4$-manifolds. If $M$ and $M'$ are homeomorphic, then for any $k \geq 0$ we have \[\Emb^\s(\sqcup_k S^1,M)\simeq \Emb^\s(\sqcup_k S^1,M').\]
\end{atheorem}

The tool in question is the embedding calculus of Goodwillie--Weiss \cite{Weiss:EPVITI,Weiss:EPVITIerratum,GoodwillieWeiss:EPVITII}, which, at the coarsest level, provides a functorial comparison map \[\Emb^\s(N,M)\lra T_\infty \Emb^\s(N,M) = \underset{k}{\mr{holim}}\, T_k\Emb^\s(N,M),\] whose target is assembled from the configuration spaces of $N$ and $M$ and maps among them (details are reviewed in Section \ref{section:embedding calculus}). According to one of the main results of the subject \cite{GoodwillieWeiss:EPVITII,GoodwillieKlein:MDSE}, this map is a weak equivalence in codimension at least 3; one says that \emph{the Taylor tower converges} to the embedding space. In particular, the theorem applies to links in 4-manifolds as in Theorem \ref{athm:knots}.

Little is known about convergence in low codimension. We begin to address this gap by proving that codimension $0$ convergence largely fails in dimension $4$.

\begin{atheorem}\label{athm:4d}
Let $M$ and $N$ be smooth, simply connected, compact $4$-manifolds. If $M$ and $N$ are homeomorphic, then $T_\infty\Emb^\s(N,M)\neq\varnothing$. In particular, if $M$ and $N$ are not also diffeomorphic, then the map \[\Emb^\s(N,M)\lra T_\infty\Emb^\s(N,M)\] is not a weak equivalence.
\end{atheorem}

In fact, we prove that there are homotopy invertible elements in $T_\infty\Emb^\s(N,M)$, which one should think of as saying that $N$ and $M$ are diffeomorphic (or at least isotopy equivalent) in the eyes of embedding calculus.

Theorems \ref{athm:knots} and \ref{athm:4d} arise from a common source. Specifically, the data involved in the constructions of embedding calculus is a pair of presheaves, one for $N$ and one for $M$. We show in Theorem \ref{thm:independence} that these presheaves are largely insensitive to smooth structure in dimension 4, and the results follow---see Section \ref{section:dim 4}.

The results above might lead one to suspect that embedding calculus is insensitive to smooth structure. The following contrasting result shows that the situation is not so simple (see Section \ref{sec:examples exotic} for further examples).

\begin{atheorem}\label{athm:spheres}
For any $n=2^j$ with $j \geq 3$, there is an exotic $n$-sphere $\Sigma$ such that $T_\infty\Emb^\s(\Sigma, S^n)=\varnothing$. In particular, the map \[\Emb^\s(\Sigma,S^n)\lra T_\infty\Emb^\s(\Sigma,S^n)\] is a weak equivalence (both sides are empty).
\end{atheorem}

Thus, embedding calculus distinguishes certain exotic spheres. Alternatively, one can interpret this as a convergence result in codimension $0$. The crucial property distinguishing the exotic spheres of Theorem \ref{athm:spheres} from $S^n$ is that they do not embed in $\mathbb{R}^{n+3}$.

To facilitate the further study of embedding calculus in the potential absence of convergence, we prove an isotopy extension theorem for $T_\infty\Emb^\s(-,-)$ (see Theorem \ref{thm:isotopy ext}). We close by demonstrating its utility with several applications.

\subsection*{Acknowledgments} The first author thanks John Francis for bringing his attention to the question of whether embedding calculus distinguishes exotic spheres. The second author would like to thank Oscar Randal-Williams for helpful conversations, and in particular the suggestion that some exotic spheres do not admit embeddings into Euclidean space with low codimension. We would also like to thank the anonymous referees, as well as Manuel Krannich and Oscar Randal--Williams for corrections to an earlier version. The first author was supported by NSF grant DMS-1906174, the Natural Sciences and Engineering Research Council of Canada (NSERC) [funding reference number 512156 and 512250], as well as the Research Competitiveness Fund of the University of Toronto at Scarborough. The second author was supported by NSF grant DMS-1803766.

\tableofcontents

\section{Preliminaries}\label{section:embedding calculus}

In this section, we gather what facts we need from the theory of embedding calculus, as well as some standard foundational material on topological manifolds. In our discussion of calculus, we adopt the perspective of \cite{BoavidaWeiss:MCHS}, but see \cite{Weiss:EPVITI,GoodwillieKleinWeiss:HSTDEC,Turchin:CFMCFMO,BoavidaWeiss:SSMCC} for other foundations.

\subsection{Embedding calculus}\label{sec:emb calc defs} Write $\MMfld^\s$ for the simplicial category whose objects are smooth manifolds without boundary, of finite type and arbitrary dimension. The morphism space $\Map_{\MMfld^s}(N,M)$ has as $n$-simplices commuting diagrams
\[\begin{tikzcd} \Delta^n \times N \arrow{rr} \arrow{rd} & & \Delta^n \times M \arrow{ld} \\
 & \Delta^n & \end{tikzcd}\]
in which the top map is a neat smooth embedding of manifolds with corners. This category is symmetric monoidal under disjoint union. 

\emph{Manifold calculus} approximates simplicial presheaves on this category by extrapolating from their values on disjoint unions of disks of a fixed dimension. More formally, let $\DDisk^\s_n \subset \MMfld^\s$ be the full subcategory on those objects that are diffeomorphic to a disjoint union of finitely many copies $\mathbb{R}^n$ with its standard smooth structure. Manifold calculus is the approximation of simplicial presheaves on $\MMfld^\s$ by simplicial presheaves on $\DDisk^\s_n$. \emph{Embedding calculus} is the application of manifold calculus to the presheaf of embeddings into a smooth manifold $M$. Fixing $n$, we write $\mathbb{E}_M^\s$ for the presheaf on $\DDisk^\s_n$ obtained by restriction of the representable presheaf on $\MMfld^\s$ determined by $M$; explicitly, we have \[\mathbb{E}_M^\s(\sqcup_k \mathbb{R}^n) \coloneqq \Emb^\s(\sqcup_k \mathbb{R}^n,M).\]
The reader is warned that our notation does not reflect the choice of $n$, which should always be clear from context.

\begin{remark}Equivalently, writing $\mathbb{E}_n^\s$ for the endomorphism operad of $\mathbb{R}^n$ with its standard smooth structure---equivalent to the framed little $n$-disks operad allowing translation, scaling, rotation, and reflection of the little disks---the simplicial category $\PSh(\DDisk^\s_n)$ of simplicial presheaves on $\DDisk_n^\s$ is equivalent to the simplicial category of right $\mathbb{E}_n^\s$-modules \cite[\S6]{BoavidaWeiss:MCHS}.
\end{remark}


An embedding $N \hookrightarrow M$ determines a map $\mathbb{E}_N^\s \to \mathbb{E}_M^\s$ of presheaves. Since the category $\DDisk^\s_n$ has a filtration by cardinality of path components, there results a canonical functorial cofiltration on mapping spaces between presheaves and localising at the objectwise weak equivalences also on the derived mapping spaces. In the situation at hand, this cofiltration is called the \emph{embedding calculus Taylor tower}.

\begin{definition}[Boavida--Weiss]\label{def:limit of tower}
Let $N$ and $M$ be smooth manifolds of dimension $n$ and $m$, respectively. The \emph{embedding calculus Taylor tower} for smooth embeddings of $N$ into $M$ is the cofiltered derived mapping space of presheaves on truncations of the simplicial category $\DDisk^\s_n$ \[T_\bullet\Emb^\s(N,M)\coloneqq \Map^h_{\PSh(\DDisk^\s_n)}(\mathbb{E}_N^\s, \mathbb{E}_M^\s).\]
\end{definition}

The cofiltered derived mapping space gives rise to a tower of comparison maps
\[\begin{tikzcd}
& & \cdots\dar\\[-5pt]
&&T_{k+1}\Emb^\s(N,M)\dar \\
\Emb^\s(N,M)\arrow{urr}{\eta_{k+1}}\arrow{rr}{\eta_k}&& T_k\Emb^\s(N,M)\dar \\[-5pt]
&&\cdots\end{tikzcd} \]
We write $T_\infty\Emb^\s(N,M)$ for the homotopy limit of the tower, which is to say the derived mapping space of presheaves on the untruncated simplicial category $\DDisk^\s_n$. One can choose a model for the derived mapping space such that (i) these constructions are functorial in $M,N \in \MMfld^\s$, (ii) there are associative and unital composition maps, (iii) both functorality and composition are compatible with the above comparison maps. See \cite[Section 3.3.1]{KupersRandalWilliams:CTSA} for further discussion of this point.

The following is \cite[Thm.~2.3]{GoodwillieWeiss:EPVITII}, relying on excision estimates from Goodwillie--Klein \cite{GoodwillieKlein:MDSE}.

\begin{theorem}[Goodwillie--Klein--Weiss]\label{thm:gkw}
The map $\eta_k$ is $(3-m + (k+1)(m-n-2))$-connected for $k>0$. In particular, if $m-n\geq3$, then $\eta_\infty$ is a weak equivalence.
\end{theorem}

In fact, we may replace $n$ in the above result by the handle dimension $\mr{hdim}(N)$ of $N$. Recall that $\mr{hdim}(N)\leq h$ if $N$ is the interior of a manifold which admits a handle decomposition with handles of index $\leq h$ only. For example, $\mr{hdim}(\mathbb{R}^n) = 0$.

If $M = N$, we write $T_\bullet \Diff(M) \subseteq T_\bullet \Emb^\s(M,M)$ for the simplicial subset of homotopy invertible maps. This distinction may very well be unnecessary; however, even in cases where every self-embedding of $M$ is a diffeomorphism, we do not know whether every path component of the limit of the Taylor tower is invertible.

\begin{question}
When are all elements of $\pi_0\,\Map^h_{\PSh(\DDisk_m^\s)}(\mathbb{E}_M^\s, \mathbb{E}_M^\s)$ invertible?
\end{question}

\subsection{Calculus for manifolds with boundary} We close with a brief description of the modifications necessary to use embedding calculus in the setting of manifolds with boundary \cite[\S9]{BoavidaWeiss:MCHS}. Fixing a smooth manifold $Z$, we write $\MMfld_Z^\s$ for the simplicial category of smooth manifolds with boundary identified with $Z$ by a diffeomorphism, and morphism spaces given by smooth embeddings that are the identity on $Z$. In particular, $\MMfld^\s=\MMfld^\s_\varnothing$. Let $\DDisk_{n,Z}^\s \subset \MMfld_Z^\s$ be the full subcategory on those objects that are diffeomorphic relative to $Z$ to a disjoint union of a collar $Z \times [0,1)$ and finitely many copies of $\mathbb{R}^n$. 


An object $P\in\MMfld^\s_Z$ determines a representable presheaf on $\MMfld_Z^\s$ and we denote its restriction to $\DDisk_{n,Z}^\s$ by $\mathbb{E}_{P,\partial}^\s$. As before, for an object $N\in\MMfld^\s_Z$ of dimension $n$, we obtain an approximation \[\Emb_\partial^\s(N,P)\lra T_\bullet\Emb_\partial^\s(N,P)=\Map^h_{\PSh(\DDisk_{n,Z}^\s)}(\mathbb{E}^\s_{N,\partial}, \mathbb{E}^\s_{P,\partial})\]
as a cofiltered derived mapping space of presheaves on $\DDisk_{n,Z}^\s$. The conclusion of Theorem \ref{thm:gkw} holds for this approximation, though handle dimension needs to be replaced by handle dimension relative to $Z$.

\subsection{A simplicial category of topological manifolds} Recall that a topological embedding $e \colon N \to M$ is \emph{locally flat} if, for every $p \in N$, there exist open neighborhoods $p\in U$ and $e(U)\subseteq V$ and homeomorphisms $U \cong \mathbb{R}^n$ and $V \cong \mathbb{R}^m$ fitting into the commuting diagram
\begin{equation}\label{eqn:locally-flat}\begin{tikzcd}N \supseteq U \dar[swap]{\cong} \rar{e|_U} & V \subseteq M \dar{\cong} \\
\mathbb{R}^n \rar{j} & \mathbb{R}^m,\end{tikzcd}\end{equation}
where $j \colon \mathbb{R}^n \to \mathbb{R}^m$ is the standard inclusion $(x_1,\ldots,x_n) \mapsto (x_1,\ldots,x_n,0,\ldots,0)$.

The simplicial category $\MMfld^\tp$ has objects topological manifolds of finite type and arbitrary dimension, with the $n$-simplices of the mapping space $\Map_{\MMfld^\tp}(N,M)$ given by commuting diagrams
\[\begin{tikzcd} \Delta^n \times N \arrow{rr} \arrow{rd} & & \Delta^n \times M \arrow{ld} \\
& \Delta^n & \end{tikzcd}\]
with the top map a locally flat embedding admitting charts as in \eqref{eqn:locally-flat} that commute with the projection to $\Delta^n$. This definition is chosen so that the isotopy extension theorem holds.

As every smooth embedding is locally flat as a consequence of the tubular neighborhood theorem, forgetting the smooth structure defines a simplicial functor from $\MMfld^\s$ to $\MMfld^\tp$.

\subsection{Microbundles} Microbundles were defined by Milnor in \cite{Milnor:MBI} and play the role of vector bundles for topological manifolds. 

\begin{definition} A \emph{retractive space} is a map $\pi \colon E\to B$ of topological spaces together with a section $\iota \colon B\to E$.\end{definition}

The spaces $E$ and $B$ are referred to as the \emph{total space} and \emph{base space}, and the maps $\pi$ and $\iota$ as \emph{projection} and \emph{zero section}. Via the zero section, we identify $B$ with its image in $E$, and we abusively refer to this image also as the zero section. We abusively refer to a retractive space simply by the letter $E$.

\begin{definition} A \emph{map $F \colon E_1 \to E_2$ of retractive spaces} is a continuous map $F \colon E_1 \to E_2$ such that the dashed filler exists in the commuting diagram
	\[\begin{tikzcd}
	B_1 \rar{\iota_1} \dar[dashed,swap]{f} & E_1 \arrow{d}{F} \rar{\pi_1} & B_1 \dar[dashed]{f} \\
	B_2 \rar{\iota_2} & E_2 \rar{\pi_2} & B_2.
	\end{tikzcd}\]
\end{definition}

Note that the map $F$ determines the map $f=\pi_2\circ F\circ\iota_1$. When we wish to emphasize the latter, we say that $F$ is a map of retractive spaces \emph{over $f$}, or \emph{over $B$} in the case $f=\id_B$. Retractive spaces and morphisms between them form a category $\Retr$.

\begin{definition}\label{def:microbundle} A \emph{microbundle} is a retractive space $E$ such that, for every $b \in B$ there is an open neighborhood $b\in U \subseteq E$ and a homeomorphism $U \cong \pi(U) \times \mathbb{R}^n$ such that the following diagram commutes:
	\[\begin{tikzcd}
		& U \arrow{dd}{\cong} \arrow{rd}{\pi} & \\
		\pi(U) \arrow{ru}{\iota} \arrow{rd} & & \pi(U) \\
		& \pi(U) \times \mathbb{R}^n \arrow{ru},& 
	\end{tikzcd}\]
where the bottom left map is induced by the inclusion of the origin and the bottom right is projection onto the first factor.
\end{definition}

\begin{example} The prototypical example of a microbundle is the tangent microbundle of a topological manifold---see Definition \ref{def:tangent bundle} below or \cite[Example (3)]{Milnor:MBI}.
\end{example}

The homeomorphisms which appear in the previous definition are called \emph{microbundle charts}. Note that, by invariance of domain, the parameter $n$ in Definition \ref{def:microbundle} is locally constant.

If $E$ is a retractive space, so is any open neighborhood $W$ of the zero section. The set of germs of maps $E_1 \to E_2$ of retractive spaces is the colimit
\[\colim_{B_1 \subseteq U \subseteq E_1} \Hom_{\Retr}(U,E_2),\]
over the poset of open subsets $U$ of $E_1$ containing $B_1$, which may be composed as follows:
\[\begin{tikzcd}\displaystyle\colim_{B_1\subseteq U\subseteq E_1}\Hom_{\Retr}(U,E_2)\times \colim_{B_2\subseteq V\subseteq E_2}\Hom_{\Retr}(V,E_3)\arrow[equals]{d}{\wr}& (F_1, F_2)\arrow[|->]{dd}\\
\displaystyle\colim_{B_1\subseteq U\subseteq E_1,\,B_2\subseteq V\subseteq E_2}\Hom_{\Retr}(U,E_2)\times \Hom_{\Retr}(V,E_3)\dar &\\
\colim_{B_1\subseteq W\subseteq E_1}\Hom_{\Retr}(W,E_3)&F_2 \circ F_1|_{F_1^{-1}(V)}.\end{tikzcd}\]
This composition is easily checked to be associative and unital.

\begin{definition}\label{def:map-of-microbundles} A \emph{map $F \colon E_1 \to E_2$ of microbundles} is a germ of a map of retractive spaces such that, for every $b \in B_1$, there are microbundle charts around $b$ and $F(b)$ fitting into the commuting diagram
	\[ \begin{tikzcd}
	U_1 \dar[swap]{\cong} \rar{F|_{U_1}} &[20pt] U_2 \dar{\cong} \\
	\pi(U_1) \times \mathbb{R}^{n_1} \rar{f|_{\pi(U_1)} \times j} & \pi(U_1) \times \mathbb{R}^{n_2},
	\end{tikzcd}\]
	where $j \colon \mathbb{R}^{n_1} \to \mathbb{R}^{n_2}$ is the standard inclusion and $f \colon B_1 \to B_2$ the map on base spaces induced by $F$.
\end{definition}

Note that maps of microbundles are fibrewise embeddings.

\begin{remark}When $E_1$ and $E_2$ are microbundles of the same fixed dimension, this definition reduces to \cite[Definition 6.3]{Milnor:MBI}.\end{remark}

\begin{example}The prototypical example of a map of microbundles is the topological derivative of a locally flat embedding---see Definition \ref{def:derivative} below.\end{example}

It is easy to check that maps of microbundles are closed under composition of germs of maps of retractive spaces, so we obtain a category $\Mic$ of numerable microbundles as a subcategory of the category $\Retr$ of retractive spaces. 

A retractive space $E$ with base $B$ can be pulled back along a continuous map $f \colon A \to B$ to give a retractive space $f^* E$ with base $A$; in the commutative diagram
\[\begin{tikzcd} A \rar \dar[swap]{f} & f^*E \rar \dar & A \dar{f} \\
B \rar{\iota} & E \rar{\pi} & B\end{tikzcd}\]
the right hand square is a pullback square, and the section $A \to f^* E$ is induced by the maps $\mr{id} \colon A \to A$ and $\iota\circ f \colon A \to E$. This exhibits a canonical map of retractive spaces $f^* E \to E$. If $E$ is a microbundle, then $f^*E$ is so as well, and the canonical map is a map of microbundles \cite[\S 3]{Milnor:MBI}. Given a microbundle $E$ with base $B$ and a topological space $X$, we let $X \times E\to X\times B$ denote the pullback of $E$ along the projection $X \times B \to B$.

Microbundles form a simplicial category $\MMic$ via the declaration \[\Map_{\MMic}(E_1,E_2)_n \coloneqq \Hom_{\Mic}\left(\Delta^n\times E_1,E_2\right).\] 
Concretely, an $n$-simplex $F \colon \Delta^n \times E_1 \to E_2$ can be described as a germ near the zero section $\Delta^n \times B_1$ of a commutative diagram
\[\begin{tikzcd} \Delta^n \times E_1 \arrow{rr}{(\pi_1,F)} \dar & & \Delta^n \times E_2 \dar \\
\Delta^n \times B_1 \arrow{rr}{(\pi_1,f)} \arrow{rd} & & \Delta^n \times B_2 \arrow{ld} \\
& \Delta^n & \end{tikzcd}\]
with the additional properties that 
\begin{enumerate}[\indent (i)]
	\item $(\pi_1,F)$ preserves the zero section, and 
	\item with respect to suitable microbundle charts, $(\pi_1,F)$ is given by the germ of $(\mr{id},f)|_{U_1} \times j \colon U_1 \times \mathbb{R}^{n_1} \to U_2 \times \mathbb{R}^{n_2}$ with $j$ the standard inclusion.
\end{enumerate}

Using that every topological horn is a retract of the corresponding topological simplex, it is easy to see that these mapping objects are Kan complexes.

Microbundles adhere to a covering homotopy theorem analogous to the classical result for vector bundles and fibre bundles, which has the following consequence. In it, $\TTop$ denotes simplicial category with objects topological spaces and morphism spaces the singular simplicial sets of mapping spaces. 

\begin{lemma}\label{lem:microbundle fibration}
	The natural map $\Map_{\MMic}(E_1, E_2)\to \Map_{\TTop}(B_1, B_2)$ is a Kan fibration with fibre over $f \colon B_1\to B_2$ canonically isomorphic to the simplicial subset of $\Map_{\MMic}(E_1, f^*E_2)$ with underlying map $\id_{B_1}$.
\end{lemma}

\begin{proof} We check that the map $\Map_{\MMic}(E_1, E_2)\to \Map_{\TTop}(B_1, B_2)$ is a Kan fibration, as the identification of the fibre is straightforward. To check the lifting property in a commutative diagram
	\[\begin{tikzcd} \Lambda^n_k \rar \dar & \Map_{\MMic}(E_1, E_2)\ \dar \\
	\Delta^n \rar & \Map_{\TTop}(B_1, B_2),\end{tikzcd}\]
	we first, by gluing, represent the top map by a map of microbundles $F \colon \Lambda^n_k \times E_1 \to E_2$ (here, and throughout, we employ the same notation for a simplicial set and its geometric realization). We similarly represent the bottom map by an extension of the map $f$ underlying $F$ to a continuous map $g \colon \Delta^n \times B_1 \to B_2$. 
	
	Let us denote by $\tilde{F}$, $\tilde{f}$, and $\tilde{g}$ the maps obtained from $F$, $f$, and $g$ using the homeomorphism of pairs 
	\[(\Delta^n,\Lambda^n_k) \cong (\Delta^{n-1} \times [0,1],\Delta^{n-1} \times \{0\}).\]
	Under this identification, the lifting problem at hand is equivalent to extending a map of microbundles $\tilde{F} \colon \Delta^{n-1} \times E_1 \to \tilde{f}^* E_2$ over $\Delta^{n-1} \times B_1$ to $\Delta^{n-1}\times [0,1] \times E_1 \to \tilde{g}^*E_2$ over $\Delta^{n-1} \times [0,1] \times B_1$. By the microbundle homotopy covering theorem \cite[Thm.~3.1]{Milnor:MBI}, there is an isomorphism of microbundles $\varphi \colon \tilde{g}^*E_2 \cong \tilde{f}^*E_2 \times [0,1]$ over $\Delta^{n-1} \times [0,1] \times B_1$. It is now evident that the desired extension exists, as we may form the product of $\tilde{F}$ with $[0,1]$ and apply $\varphi^{-1}$.
\end{proof}

\subsection{Topological tangency}\label{sec:top-tangent}

We come now to the motivating example of a microbundle, the ``tangent bundle'' of a topological manifold \cite[Lemma 2.1]{Milnor:MBI}.

\begin{definition}\label{def:tangent bundle}
	Let $M$ be a topological manifold. The \emph{topological tangent bundle} of $M$, denoted $T^\tp M$, is the retractive space \[M\xrightarrow{\Delta} M\times M\xrightarrow{\pi} M,\] where $\pi$ is the projection onto the first factor
\end{definition}

To verify that $T^\tp M$ is a microbundle, it suffices by locality to assume $M=\mathbb{R}^m$, in which case we may appeal to the commuting diagram
\[\begin{tikzcd}
		& \mathbb{R}^m\times\mathbb{R}^m \arrow{dd}{\cong} \arrow{rd}{\pi} &&(x,y)\ar[mapsto]{dd} \\
		\mathbb{R}^m \arrow{ru}{\Delta} \arrow{rd} & & \mathbb{R}^m \\
		& \mathbb{R}^m \times \mathbb{R}^m \arrow{ru},& &(x,y-x).
	\end{tikzcd}\] 

A smooth embedding has a derivative, and likewise a locally flat embedding $\varphi \colon N\to M$ has a topological derivative.

\begin{definition}\label{def:derivative} Let $\varphi \colon N \to M$ be a locally flat embedding. The \emph{topological derivative} $T^\tp \varphi \colon T^\tp N\to T^\tp M$ of $\varphi$ is the map of microbundles
	\[\begin{tikzcd} N \rar{\Delta} \dar[swap]{\varphi} & N \times N \rar{\pi} \dar{\varphi \times \varphi} & N \dar{\varphi} \\
	M \rar{\Delta} & M \times M \rar{\pi} & M.\end{tikzcd}\]
\end{definition}

To verify that $T^\tp\varphi$ is a map of microbundles, we may by locality assume that $\varphi$ is the standard inclusion $\mathbb{R}^n \hookrightarrow \mathbb{R}^m$, in which case the bundle chart constructed above implies the claim. Thus, we obtain a simplicial functor $T^\tp \colon \MMfld^\tp \to \MMic$

\subsection{Comparing tangent bundles} We write $\VVec$ for the simplicial category of numerable vector bundles and maps of vector bundles, which for us are always fibrewise linear injections. Specifically, given vector bundles $E_1\to B_1$ and $E_2\to B_2$, an $n$-simplex of $\Map_{\VVec}(E_1, E_2)$ is a commuting diagram \[
\begin{tikzcd}
E_1\times\Delta^n\ar{d}\ar{rr}&&E_2\times\Delta^n\ar{d}\\
B_1\times\Delta^n\ar{dr}\ar{rr}&&B_2\times\Delta^n\ar{dl}\\
&\Delta^n
\end{tikzcd}\] in which the top map is a fibrewise linear injection. As before, these mapping spaces are Kan complexes.

We record the following standard consequence of the covering homotopy theorem for vector bundles \cite[Theorem 4.3]{Husemoller:FB}, whose proof proceeds along the lines of Lemma \ref{lem:microbundle fibration}.

\begin{lemma}\label{lem:bundle fibration}
	The natural map $\Map_{\VVec}(E_1, E_2)\to \Map_{\TTop}(B_1, B_2)$ is a Kan fibration with fibre over $f \colon B_1\to B_2$ canonically isomorphic to the simplicial subset of $\Map_{\VVec}(E_1, f^*E_2)$ with underlying map $\id_{B_1}$.
\end{lemma}

Vector bundles are in particular microbundles, and assigning to a vector bundle its underlying microbundle extends to a simplicial functor $\MMic \to \VVec$.

We now have two ways of extracting a microbundle from a smooth manifold $M$: first, by considering its tangent bundle $TM$ as a microbundle; second, by forgetting the smooth structure and considering $T^\tp M$. To compare these, we use the following construction.

\begin{construction}\label{construction:metric micro map}
	Fix a Riemannian metric on the smooth manifold $M$. The $t=1$ exponential map is defined on a neighborhood $U$ of the zero section, and the assignment \begin{align*}
	\exp_M \colon TM\supseteq U&\lra T^\tp M\\
	(p,v)&\longmapsto (p, \exp(p,v))
	\end{align*}
	defines a map of retractive spaces.
\end{construction}

The following is \cite[Theorem 2.2]{Milnor:MBI}.

\begin{proposition}[Milnor]\label{thm:metric micro iso}The map of Construction \ref{construction:metric micro map} defines an isomorphism of microbundles $TM \overset{\sim}\to T^\tp M$.
\end{proposition}

\section{Formally smooth manifolds}\label{section:calc and formal smooth}

The first goal of this section is to factor the forgetful functor from smooth to topological manifolds as in the commuting diagram
\[\label{eqn:mfd-cats-zigzag} \begin{tikzcd} \MMfld^\s \rar & \MMfld^\tp \\
\MMfld^\r \uar{\simeq} \rar & \MMfld^\f \uar. \end{tikzcd}\] The simplicial category $\MMfld^\r$ is a category of Riemannian manifolds under embeddings respecting the metric up to specified homotopy. It is introduced because Construction \ref{construction:metric micro map} requires a Riemannian metric. As a result of the homotopy equivalence between $\mr{O}(n)$ and $\mr{GL}(n)$, the leftmost functor is an equivalence, and the role of $\MMfld^\r$ is as a convenient proxy for $\MMfld^\s$. The simplicial category $\MMfld^\f$ is a category of \emph{formally smooth} manifolds, which is to say manifolds equipped with vector bundle refinements of their topological tangent bundles. 

The second goal of this section is to prove Theorem \ref{thm:independence}, which asserts that all information detectable by embedding calculus is contained in $\MMfld^\f$.

\subsection{Simplicial categories of Riemannian and formally smooth manifolds} \label{sec:simpl-cats} In this section, we have in mind the model of the homotopy pullback of simplicial categories explained in Appendix \ref{section:construction} following \cite{Andrade:FMIEA}.

We begin with the construction of $\MMfld^\r$. We write $\MMet$ for the simplicial category whose objects are vector bundles endowed with Riemannian metrics and whose morphisms are fibrewise linear isometries, which are assembled into simplicial sets in the same manner as in $\VVec$. As before, these mapping spaces are Kan complexes and we have the following consequence of local triviality.

\begin{lemma}\label{lem:bundle fibration met}
	The natural map $\Map_{\MMet}(E_1, E_2)\to \Map_{\TTop}(B_1, B_2)$ is a Kan fibration with fibre over $f \colon B_1\to B_2$ canonically isomorphic to the simplicial subset of $\Map_{\MMet}(E_1, f^*E_2)$ with underlying map $\id_{B_1}$.
\end{lemma}

There is a canonical simplicial forgetful functor from $\MMet$ to $\VVec$.

\begin{proposition}\label{cor:bundles are metric}
	The forgetful functor $\MMet\to \VVec$ is essentially surjective and induces weak equivalences on mapping spaces.
\end{proposition}

\begin{proof}
	The first claim follows from the fact that every numerable vector bundle admits a Riemannian metric. For the second claim, it suffices by Lemmas \ref{lem:bundle fibration} and \ref{lem:bundle fibration met} to show that the maps induced on point-set fibres in the commuting diagram
	\[\begin{tikzcd}
	\Map_{\MMet}(E_1,E_2) \rar \dar & \Map_{\VVec}(E_1,E_2) \dar \\
	\Map_{\TTop}(B_1, B_2) \rar[equals] & \Map_{\TTop}(B_1, B_2),\end{tikzcd}\]
	are weak equivalences. By the same results, we may identify the lefthand fibre over $f$ with the singular simplicial set of the space of sections of the associated bundle of non-compact Stiefel manifolds, whose fibres are general linear groups (resp.\~righthand, compact, orthogonal). The conclusion then follows as the inclusion of the orthogonal group into the general linear group is a homotopy equivalence.
\end{proof}

We use this to define $\MMfld^\r$, which is a homotopy pullback as in Appendix \ref{section:construction}.

\begin{definition}\label{def:riemannian}
	The \emph{simplicial category of Riemannian manifolds} is the homotopy pullback in the diagram \[\begin{tikzcd} \MMfld^\r\ar[r]\ar[d]&\MMfld^\s \dar{T}\\
	\MMet\ar[r]&\VVec \end{tikzcd}\] of simplicial categories over $\TTop$.
\end{definition}

\begin{notation}
	We denote the morphism spaces in $\MMfld^\r$ by $\Emb^\r(-,-)$.
\end{notation}

Thus, an object of $\MMfld^\r$ is a smooth manifold with a choice of metric, and a morphism is a fibrewise isometry covering a smooth embedding, together with a fibrewise homotopy through linear injections to the derivative of the embedding. 

As the following result illustrates, the forgetful functor exhibits $\MMfld^\r$ as a proxy for $\MMfld^\s$. This proxy is easier to map out of. 

\begin{proposition}\label{prop:smooth and riemannian}
	The forgetful functor $\MMfld^\r\to \MMfld^\s$ is essentially surjective and induces weak equivalences on mapping spaces.
\end{proposition}

\begin{proof}
The first claim follows from the statement that every smooth manifold admits a Riemannian metric. The second claim follows from Proposition \ref{prop:homotopy pullback}, Lemma \ref{lem:bundle fibration}, and Corollary \ref{cor:bundles are metric}.
\end{proof}

We continue with the construction of $\MMfld^\f$, which is a homotopy pullback as in Appendix \ref{section:construction}.

\begin{definition}\label{def:formally smooth}
	The \emph{simplicial category of formally smooth manifolds} is the homotopy pullback in the diagram \[\begin{tikzcd} 
	\MMfld^\f\ar[r]\ar[d]&\MMfld^\tp \dar{T^\tp}\\
	\VVec\ar[r]&\MMic\end{tikzcd}\] of simplicial categories over $\TTop$.
\end{definition}

\begin{notation}
	We denote the morphism spaces in $\MMfld^\f$ by $\Emb^\f(-,-)$.
\end{notation}

Thus, an object of $\MMfld^\f$ is a topological manifold with a vector bundle refinement of its topological tangent bundle, and a morphism is a fibrewise linear injection covering a topological embedding, together with a fibrewise homotopy through embeddings to the topological derivative of the embedding.

It remains to construct the functor $\MMfld^\r \to \MMfld^\f$.

\begin{construction}\label{construction:example functor}
	We obtain $\MMfld^\r\to \MMfld^\f$ as an instance of Construction \ref{construction:functor}. The requisite data are the following.
	\begin{enumerate}
		\item The simplicial functor $\MMfld^\r\to\MMet\to \VVec$ associating to a Riemannian manifold its tangent bundle.
		\item The simplicial functor $\MMfld^\r \to \MMfld^\s\to \MMfld^\tp$ associating to a Riemannian manifold its underlying topological manifold.
		\item The natural isomorphism indicated by the thick arrow between left-bottom and top-right compositions in the diagram \[\begin{tikzcd} \MMfld^\r\arrow{dd}\arrow{rr}&[-10pt]&[-10pt]\MMfld^\s\arrow{dd}\\[-5pt]\\[-5pt]
		\VVec\arrow[Rightarrow]{uurr}{\cong} \arrow{rr} && \MMic\end{tikzcd}\] arising from Construction \ref{construction:metric micro map}.
	\end{enumerate}
\end{construction}

\begin{remark}
Upon restricting to the respective full subcategories of manifolds of dimension different from $4$, the functor of Construction \ref{construction:example functor} becomes an equivalence by smoothing theory \cite[Essays IV, V]{KirbySiebenmann:FETMST}.
\end{remark}

\subsection{Smooth embeddings of Euclidean spaces}\label{section:smooth euclidean embeddings} In the next sections, we assemble results on various types of embeddings of Euclidean spaces, which will be used below in the proof of Theorem \ref{thm:independence}. We begin in the smooth context, where these results are standard, but we include proofs for the sake of completeness. 

Fix a smooth $m$-manifold $M$ and a natural number $0<n\leq m$, as well as a point $p \in M$. We introduce four simplicial sets, the first three defined as pullbacks of diagrams of the form \[\begin{tikzcd}
	&[-5pt] X\ar{d}\\[-5pt]
	\{p\}\ar{r}&M.
\end{tikzcd}\]

\begin{enumerate}
	\item Taking $X=\Map_{\VVec}(T\mathbb{R}^n,TM)$ mapping to $M$ by evaluation at the origin followed by projection, we obtain $\Map_{\VVec,p}(T\mathbb{R}^n,TM)$.
	\item Taking $X=\Map_{\VVec}(T_0\mathbb{R}^n,TM)$ mapping to $M$ in the same way, we obtain $\Map_{\VVec,p}(T_0\mathbb{R}^n,TM)$, otherwise known as the (non-compact) Stiefel manifold of $n$-planes in $T_p M$.
	\item Fixing an open subset $0\in U \subseteq \mathbb{R}^n$, and taking $X=\Emb^\s(U,M)$ mapping to $M$ by evaluation at the origin, we obtain  $\Emb_p^\s(U,M)$.
	\item Finally, we write $G^\s_p(n,M) \coloneqq \colim_{0\in U\subseteq \mathbb{R}^n}\Emb_p^\s(U, M)$ for the simplicial set of germs of smooth embeddings (here $U$ ranges over open subsets containing the origin).
\end{enumerate}

\begin{lemma}\label{lem:smooth germ} All maps in the commuting diagram \[\begin{tikzcd} \Emb_p^\s(\mathbb{R}^n,M) \dar \rar & G^\s_p(n,M) \dar \\
	\Map_{\VVec,p}(T\mathbb{R}^n,TM) \rar & \Map_{\VVec,p}(T_0\mathbb{R}^n,TM)\end{tikzcd}\]
	are weak equivalences.
\end{lemma}
\begin{proof}
For the top map, we note that the restriction $\Emb_p^\s(\mathbb{R}^n, M)\to \Emb_p^\s(U,M)$ is a weak equivalence whenever $U$ is an open ball centered at the origin, since the inclusion $U\subseteq\mathbb{R}^n$ isotopic to a diffeomorphism relative to the origin. The claim now follows from observation that the subposet of such open balls is final in the poset of all open neighborhoods of the origin, and both are filtered. For the bottom map, the claim is a consequence of the contractibility of $\mathbb{R}^n$ and the homotopy covering theorem. For the righthand map, the claim may be tested on compact families of germs, so we may assume that $M=\mathbb{R}^m$. In this case, the Stiefel manifold includes canonically into $\Emb_p^\s(\mathbb{R}^n, \mathbb{R}^m)$, and composing with the map to $G_p^\s(n,\mathbb{R}^m)$ supplies a homotopy inverse. For the lefthand map, the claim follows by two-out-of-three.
\end{proof}

We will also have use for a mild generalization of the claim regarding the top map. Let $N=\sqcup_{i\in I} \mathbb{R}^{n_i}$ for some finite set $I$, and fix a collection $p_i\in M$ of points for each $i\in I$ such that $p_i\neq p_j$ if $i\neq j$. Let $\Emb^\s_{p_I}(N,M)\subseteq\Emb^\s(N,M)$ be the simplicial subset of embeddings sending the origin in $\mathbb{R}^{n_i}$ to $p_i$. 

\begin{lemma}\label{lem:multiple germs}
The canonical map $\Emb_{p_I}^\s(N,M)\to \prod_{i\in I}G_{p_i}^\s(n_i,M)$ is a weak equivalence.
\end{lemma}
\begin{proof}
The map in question factors through $G^\s_{p_I}(n_I,M)\coloneqq\colim_{U\subseteq N}\Emb_{p_I}^\s(U, M)$, where $U$ ranges over open subsets containing the origin in $\mathbb{R}^{n_i}$ for every $i\in I$. As in the previous argument, the subposet consisting of the disjoint unions of open balls around the respective origins is final, and both are filtered. Thus, since the inclusion of such an open set into $N$ is isotopic to a diffeomorphism, the first map is a weak equivalence. On the other hand, the map \[G^\s_{p_I}(n_I,M)\lra \prod_{i\in I} G^\s_{p_i}(n_i,M)\] is an isomorphism; indeed, injectivity is immediate, and surjectivity follows from the observation that any family of $I$-tuples of germs parametrized over a compact space (such as a simplex) can be represented by a family of $I$-tuples of embeddings whose images are pairwise disjoint at every point of the parameter space. 
\end{proof}

We write $\Conf_I(M) \coloneqq \{(p_i)_{i \in I} \mid p_i \neq p_j \text{ if $i \neq j$}\} \subseteq M^I$ for the \emph{configuration space} of particles in $M$ labeled by $I$.

\begin{proposition}\label{prop:smooth pullback}
	Let $M$ be a smooth manifold and $N=\sqcup_{i\in I} \mathbb{R}^{n_i}$. The diagram \[\begin{tikzcd}
	\Emb^\s(N,M)\ar{r}\ar{d}&\Map_{\VVec}(TN,TM)\ar{d}\\
	\Conf_I(M)\ar{r}&M^I
	\end{tikzcd}\] induced by evaluation at the respective origins is homotopy Cartesian.
\end{proposition}

\begin{proof}This square is the outer square in the commuting diagram
	\[\begin{tikzcd}
	\Emb^\s(N,M)\ar{d}\ar{r}&\prod_{i\in I}\Emb^\s(\mathbb{R}^{n_i},M)\ar{d}\ar{r}&\prod_{i\in I}\Map_{\VVec}(T\mathbb{R}^{n_i}, TM)\ar{d}\\
	\Conf_I(M)\ar{r}&M^I\ar[equals]{r}&M^I,
	\end{tikzcd}\] it suffices to verify that each of the inner squares is homotopy Cartesian. The left two vertical maps are fibrations by the isotopy extension theorem \cite[Chapter 6]{Wall:DT}, and the righthand vertical map is a product of fibrations, the $i$th map being the composite of the two fibrations \[\Map_{\VVec}(T\mathbb{R}^{n_i}, TM) \to \Map_{\TTop}(\mathbb{R}^{n_i},M) \to \Map_{\TTop}(\{0\},M).\] Thus, it suffices to establish that the induced maps on fibres are weak equivalences.
	
	For the righthand square, the map on fibres is a product of weak equivalences by Lemma \ref{lem:smooth germ}. For the lefthand square, the map on fibres is the top map in the commuting diagram	\[\begin{tikzcd} \Emb^\s_{p^I}(N,M) \rar \dar & \prod_{i\in I}\Emb^\s_{p_i}(\mathbb{R}^{n_i},M) \dar \\
	\prod_{i\in I}G_{p_i}^\s(n_i,M) \rar[equals] & \prod_{i\in I}G_{p_i}^\s(n_i,M), \end{tikzcd} \] and the vertical maps are weak equivalences by Lemmas \ref{lem:smooth germ} and \ref{lem:multiple germs}.

\end{proof}

\subsection{Topological embeddings of Euclidean spaces}\label{section:topological euclidean embeddings} We turn now to the topological versions of these facts, our goal being a description of locally flat embeddings of Euclidean spaces in terms of microbundle maps and configuration spaces. 

Taking $M$ instead to be merely a topological manifold, we define the simplicial sets $\Map_{\MMic,p}(T\mathbb{R}^n,TM)$, $\Map_{\MMic,p}(T_0\mathbb{R}^n,TM)$, $\Emb_p^\tp(U,M)$, and $G^\tp_p(n,M)$ by replacing smooth embeddings with locally flat embeddings and vector bundles with microbundles in the definitions of the previous section.

\begin{lemma}\label{lem:topological germ} All maps in the commuting diagram \[\begin{tikzcd} \Emb_p^\tp(\mathbb{R}^n,M) \dar \rar & G^\tp_p(n,M) \dar \\
	\Map_{\MMic,p}(T\mathbb{R}^n,TM) \rar & \Map_{\MMic,p}(T_0\mathbb{R}^n,TM)\end{tikzcd}\]
	are weak equivalences. In fact, the righthand vertical map is an isomorphism.
\end{lemma}
\begin{proof}
The claim regarding the righthand map follows upon inspecting the definitions, and the same argument as in Lemma \ref{lem:smooth germ} suffices for the remaining three. 
\end{proof}

As in the smooth case, extending our notation in the obvious way, we have the following generalization.

\begin{lemma}\label{lem:topological multiple germs}
	The canonical map $\Emb_{p_I}^\tp(N,M)\to \prod_{i\in I}G_{p_i}^\tp(n_i,M)$ is a weak equivalence.
\end{lemma}

Given these inputs and isotopy extension for locally flat embeddings \cite{EdwardsKirby:DSI} (see \cite[Theorem 6.17]{Siebenmann:DOHOSS} for the variant with parameters), the topological analogue of Proposition \ref{prop:smooth pullback} follows by the same argument.

\begin{proposition}\label{prop:topological pullback}
	Let $M$ be a topological manifold and $N=\sqcup_{i\in I} \mathbb{R}^{n_i}$. The diagram \[\begin{tikzcd}
	\Emb^\tp(N,M)\ar{r}\ar{d}&\Map_{\MMic}(T^\tp N,T^\tp M)\ar{d}\\
	\Conf_I(M)\ar{r}&M^I
	\end{tikzcd}\] induced by evaluation at the respective origins is homotopy Cartesian.
\end{proposition}

\subsection{Formally smooth embeddings of Euclidean spaces} The key calculation in the proof of Theorem \ref{thm:independence} is a comparison between Riemannian embeddings and formally smooth embeddings. We start with a lemma concerning Riemannian embeddings.

\begin{lemma}\label{lem:riemannian pullback}
	Let $M$ be a Riemannian manifold and $N=\sqcup_{i\in I} \mathbb{R}^{n_i}$. The diagram \[\begin{tikzcd}
	\Emb^\r(N,M)\ar{r}\ar{d}&\Map_{\MMet}(TN,TM)\ar{d}\\
	\Conf_I(M)\ar{r}&M^I
	\end{tikzcd}\] induced by evaluation at the respective origins is homotopy Cartesian.
\end{lemma}

\begin{proof}
	The upper square of the commuting diagram
	\[\begin{tikzcd}
	\Emb^\r(N,M)\ar{r}\ar{d}&\Map_{\MMet}(TN, TM)\ar{d}\\
	\Emb^\s(N,M)\ar{r}\ar{dd}&\Map_{\VVec}(TN,TM)\ar{d}\\
	&\Map_{\TTop}(N,M)\ar[d]\\
	\Conf_I(M)\ar{r}&M^I
	\end{tikzcd}\] is homotopy Cartesian by Propositions \ref{cor:bundles are metric} and \ref{prop:smooth and riemannian} (they imply that right, resp.~left, vertical map is a weak equivalence), and the bottom square is homotopy Cartesian by Proposition \ref{prop:smooth pullback}.
\end{proof}

We come now to the result of interest.

\begin{proposition}\label{prop:euclidean embeddings}
Let $M$ be a Riemannian manifold and $N=\sqcup_{i\in I}\mathbb{R}^{n_i}$. Then the canonical map\[\Emb^\r(N,M) \lra \Emb^\f(N,M)\] is a weak equivalence.\end{proposition}
\begin{proof}Consider the following diagram:
	\[\begin{tikzcd} \Emb^\r(N,M)\rar \dar &\Emb^\f(N,M)\rar \dar &\Emb^\tp(N,M)\dar \\
	\Map_{\MMet}(TN,TM)\rar &\Map_{\VVec}(TN,TM)\rar &\Map_{\MMic}(T M, T N) .\end{tikzcd}\]
	The right square commutes, but the left square commutes only up to specified homotopy.
	
	The maps from $\Map_{\VVec}(TN, TM)$ and $\Map_{\MMic}(TN, TM)$ to $\Map_{\TTop}(N,M)$ are fibrations by Lemma \ref{lem:bundle fibration}, so Proposition \ref{prop:homotopy pullback} grants that the righthand square is homotopy Cartesian. Therefore, since the lower lefthand map is a weak equivalence by Corollary \ref{cor:bundles are metric}, it suffices to show that the outer diagram is also homotopy Cartesian. By Proposition \ref{prop:topological pullback} and Lemma \ref{lem:riemannian pullback}, the vertical homotopy fibres in the outer diagram are compatibly identified with the homotopy fibre of the inclusion $\Conf_I(M)\subseteq M^I$, and the claim follows.\end{proof}

\subsection{Consequences for embedding calculus} 

In order to state the main result, we extend our notation in the obvious way by writing $\DDisk_n^\f$ and $\DDisk_n^r$ for the full subcategories on disjoint unions of finitely many copies of $\mathbb{R}^n$ in the appropriate categories of manifolds, and similarly for derived mapping spaces of simplicial presheaves on these categories.

\begin{theorem}\label{thm:independence}
Given Riemannian metrics on smooth manifolds $M$ and $N$, there is a canonical weak equivalence \[T_\bullet\Emb^\s(N,M)\simeq \Map^h_{\PSh(\DDisk^\f_n)}(\mathbb{E}^\f_N,\mathbb{E}_M^\f).\] In particular, the embedding calculus Taylor tower depends only on $M$ and $N$ as formally smooth manifolds.
\end{theorem}

\begin{remark}The choice of Riemannian metric on $M$ and $N$ is irrelevant; the space of Riemannian metrics on a smooth manifold is contractible and our constructions are continuous in the Riemannian metric in the sense that a path of Riemannian metrics gives rise to a homotopy of zigzags of maps between the left and right hand side.\end{remark}

\begin{remark}
Similar methods serve to establish a version of Theorem \ref{thm:independence} for manifolds $M$ and $N$ with common boundary $Z$.
\end{remark}

The theorem is an immediate consequence of the following result, which will follow easily from Proposition \ref{prop:euclidean embeddings}. Write $f \colon \DDisk_n^\r\to \DDisk_n^\s$ and $g \colon \DDisk_n^r \to \DDisk_n^\f$ for the respective forgetful functors, and write $\Phi\coloneqq\mathbb{L}g_!f^*$ for the composite of the (automatically derived) restriction and derived induction functors pertaining to these maps (a concrete model for the latter is available via a functorial cofibrant replacement, for example).

\begin{proposition}
Fix $n\geq0$.
\begin{enumerate}
\item The functor $\Phi \colon \PSh(\DDisk_n^\s)\to \PSh(\DDisk_n^\f)$ is essentially surjective up to weak equivalence and induces weak equivalences on derived mapping spaces.
\item For any Riemannian manifold $M$, there is a canonical weak equivalence $\Phi(\mathbb{E}_M^\s)\simeq \mathbb{E}_M^\f$.
\end{enumerate}
\end{proposition}
\begin{proof}
By Proposition \ref{prop:euclidean embeddings}, the functors $f$ and $g$ are Dwyer--Kan equivalences and hence so are the induced maps on presheaf categories \cite{Korschgen:DKEIETEP}, implying the first claim. For the second, we observe the zig-zag \[
\Phi(\mathbb{E}^\s_M)=\mathbb{L}g_!f^*\mathbb{E}^\s_M\xleftarrow{\sim} \mathbb{L}g_!\mathbb{E}_M^\r \xrightarrow{\sim}\mathbb{L}g_!g^*\mathbb{E}_M^\f\xrightarrow{\sim} \mathbb{E}_M^\f,\] where the first two weak equivalences follow from Proposition \ref{prop:euclidean embeddings}.
\end{proof}

\section{Embedding calculus in dimension 4}\label{section:dim 4}

The goal of this section is to prove Theorems \ref{athm:knots} and \ref{athm:4d}. At the heart of the matter is the question of deciding when two 4-manifolds are formally diffeomorphic. 

\subsection{Formal diffeomorphisms of $4$-manifolds} To a homeomorphism $\varphi \colon N \to M$ between formally smooth $4$-manifolds, there is associated an element $\mr{ks}(\varphi) \in H^3(N;\mathbb{Z}/2)$ called the \emph{relative Kirby--Siebenmann invariant} (in higher dimensions, it is sometimes called the \emph{Casson--Sullivan invariant} \cite{Ranicki:OTH}). A definition for smooth $4$-manifolds is given in \cite[Corollary 8.3D]{FreedmanQuinn:TO4M} and we define it now for formally smooth $4$-manifolds.

By \cite[Corollary 2]{Kister:MBFB}, the topological tangent bundles of $N$ and $M$ have essentially unique lifts to $\mathbb{R}^4$-bundles with structure group $\mr{Top}(4)$, where we recall that $\mr{Top}(4)$ is the space of self-homeomorphisms of $\mathbb{R}^4$. Thus, we have the following diagram \[\begin{tikzcd} 
&&&[5pt] B\mr{O}(4)\arrow{dd} \\
&&M\arrow{rd}{T^\tp M}\arrow{ur}{T^\f M}\\
N\ar[bend left=35]{uurrr}{T^\f N}\ar{urr}{\varphi}\ar{rrr}{T^\tp N}&&&B\mr{Top}(4),\end{tikzcd}\] in which the righthand and bottom triangles may be taken to commute strictly and the outer triangle to commute up to homotopy. The obstruction to the remaining 3-dimensional cell of the diagram commuting up to homotopy is the homotopy class of a map from $N$ to $\mr{Top}(4)/\mr{O}(4)$, which is an Eilenberg--Mac Lane space $K(\mathbb{Z}/2\mathbb{Z},3)$ through dimension $5$ \cite[Theorems 8.3B and 8.7A]{FreedmanQuinn:TO4M}. By definition, the resulting obstruction class in $H^3(N;\mathbb{Z}/2\mathbb{Z})$ is $\mr{ks}(\varphi)$. The following is immediate:

\begin{proposition}\label{prop:ks invariant} Suppose $\varphi \colon N \to M$ is a homeomorphism between formally smooth $4$-manifolds. Then $\mr{ks}(\varphi) \in H^3(N;\mathbb{Z}/2)$ vanishes if and only if $\varphi$ lifts to an isomorphism between $N$ and $M$ in $\MMfld^\f$.\end{proposition}

\begin{corollary}\label{cor:4d formal diff}
Let $N$ and $M$ be smooth, simply connected, compact $4$-manifolds. If $N$ and $M$ are homeomorphic, then $N$ and $M$ are isomorphic in $\MMfld^\f$.
\end{corollary}
\begin{proof}
Choosing a homeomorphism $\varphi$, we have $\mr{ks}(\varphi) \in  H^3(N;\mathbb{Z}/2\mathbb{Z})=0$ by Poincar\'{e} duality and the assumption that $N$ is simply-connected.
\end{proof}

\begin{remark} Supposing $N$ and $M$ to be smooth, the \emph{sum-stable smoothing theorem} in \cite[Section 8.6]{FreedmanQuinn:TO4M} asserts that, if $\varphi$ lifts to an isomorphism between $N$ and $M$ in $\MMfld^\f$ then $N$ and $M$ are \emph{stably diffeomorphic}: there exists $g \geq 0$ and a diffeomorphism
	\[\tilde{\varphi} \colon N \#_g (S^2 \times S^2) \xrightarrow{\simeq} M \#_g (S^2 \times S^2).\]
	The converse is also true, as forming the connected sum with $S^2 \times S^2$ does not affect the value of the relative Kirby--Siebenmann invariant.\end{remark}

\subsection{Proof of Theorems \ref{athm:knots} and \ref{athm:4d}}

The proofs of these theorems are now a matter of stringing weak equivalences together.

\begin{proof}[Proof of Theorem \ref{athm:knots}]
Assuming that $N$ and $M$ are smooth, simply connected, compact smooth 4-manifolds which are homeomorphic, we have the equivalences \begin{align*}
\Emb^\s(\sqcup_k S^1,N)&\simeq T_\infty\Emb^\s(\sqcup
_k S^1,N)&\quad \text{(\ref{thm:gkw})},\\
&\simeq  \Map^h_{\PSh(\DDisk^\f_1)}(\mathbb{E}_{\sqcup
_kS^1}^\f,\mathbb{E}_N^\f)&\quad \text{(\ref{thm:independence})},\\
&\simeq \Map^h_{\PSh(\DDisk^\f_1)}(\mathbb{E}_{\sqcup
_kS^1}^\f,\mathbb{E}_M^\f)&\quad \text{(\ref{cor:4d formal diff})},\\
&\simeq T_\infty\Emb^\s(\sqcup
_kS^1,M)&\quad  \text{(\ref{thm:independence})},\\
&\simeq \Emb^\s(\sqcup
_kS^1, M)&\quad\text{(\ref{thm:gkw})}.
\end{align*}
\end{proof}

\begin{proof}[Proof of Theorem \ref{athm:4d}]
Once more assuming that $N$ and $M$ are smooth, simply connected, compact smooth 4-manifolds, we have \[
T_\infty\Emb^\s(N,M)\overset{\text{(\ref{thm:independence})}}{\simeq} \Map^h_{\PSh(\DDisk^\f_4)}(\mathbb{E}_N^\f,\mathbb{E}_M^\f)
\overset{\text{(\ref{cor:4d formal diff})}}{\simeq} \Map^h_{\PSh(\DDisk^\f_4)}(\mathbb{E}_N^\f,\mathbb{E}_N^\f),\] and this last space is non-empty, as it contains the identity. On the other hand, any embedding of $N$ into $M$ is a diffeomorphism by compactness, so $\Emb^\s(N,M)$ is non-empty if and only if $N$ and $M$ are diffeomorphic.
\end{proof} 

\begin{remark}
Our proof of Theorem \ref{athm:knots} implies that, under the same hypotheses, the finite stages $T_r \Emb^\s(\sqcup_k S^1,N)$ and $T_r\Emb^\s(\sqcup_k S^1,M)$ are also weakly equivalent. A related result appears in \cite[Theorem A]{AroneSzymik:SKCESS}, where a study of the second stage of the Taylor tower is leveraged to show that, if $N$ is $n$-dimensional, the $(2n-7)$-skeleton of $\Emb^\s(S^1,N)$ does not depend on the smooth structure of $N$.
\end{remark}

\begin{remark}Note that the element of $T_\infty\Emb^\s(N,M)$ obtained in the course of the proof of Theorem \ref{athm:4d} is homotopy-invertible.\end{remark}

\subsection{Remarks on the study of smooth 4-manifolds}

In this section, we discuss some expected consequences of our results for the study of smooth 4-manifolds. This discussion is informal and should be taken as motivation for further investigation.

One source of invariants for smooth manifolds are configuration space integrals. Pioneered by Kontsevich \cite{Kontsevich:FDLT} and developed subsequently by many authors, this type of invariant is given schematically by a map of the form \[H^*(\Gamma)\lra H^*(\Emb^\s(N,M)),\] where $\Gamma$ is a combinatorially defined cochain complex of graphs. We will remain vague about the coefficients and the precise flavor of graph complex in question (there are many options); suffice it to say that an element of the graph complex is typically interpreted as a set of instructions for combining differential forms on compactified configuration spaces.

Extrapolating from results in the literature, such as \cite{Volic:FTKICF} and \cite{Prigge:TCFBSEC}, a positive answer to the following general question is expected.

\begin{question}\label{question:integrals}
Do configuration space integrals factor through the limit of the embedding calculus Taylor tower?
\[\begin{tikzcd}
H^*(\Gamma)\ar[dashed]{dr}[swap]{\exists?}\ar{r}&H^*(\Emb^\s(N,M))\\
&H^*(T_\infty\Emb^\s(N,M))\ar{u}
\end{tikzcd}\]
\end{question}

If Question \ref{question:integrals} has a positive answer, Theorem \ref{thm:independence} implies that these invariants cannot distinguish exotic smooth structures on $M$ by taking $N$ to be homeomorphic but not diffeomorphic to $M$, unless they are already not formally diffeomorphic. For example, it would follow that this use of configuration space integrals can shed no light on the smooth Poincar\'{e} conjecture in dimension $4$, or at least not directly.

A second use for configuration space integrals, accessed by setting $M=N$, is to study the classifying spaces of diffeomorphism groups. Again assuming a positive answer to Question \ref{question:integrals}, Theorem \ref{thm:independence} implies that this approach is limited to detecting the algebraic topology of \emph{formal} diffeomorphism groups; for example, the results of Watanabe \cite{Watanabe:ENERHGD} on the rational homotopy of $B\Diff_\partial(D^4)$ should be interpreted as results about the automorphisms of $D^4$ as a formally smooth manifold. This change in perspective has concrete consequences.

\begin{proposition}
If Question \ref{question:integrals} has a positive answer, then the natural map \[\Top(4)/\mr{O}(4) \to \Top/\mr{O}\] is not a weak equivalence, even after rationalising.
\end{proposition}

\begin{proof}
By \cite[Thm.~1.1]{Watanabe:ENERHGD}, configuration space integrals produce many nontrivial classes of positive degree in $H^*(B\Diff_\partial(D^4);\mathbb{R})$, which our assumption implies are pulled back from $H^*(BT_\infty \Diff_\partial(D_4);\mathbb{R})$. A version of Theorem \ref{thm:independence} with boundary implies that the map $\Diff_\partial(D^4) \to T_\infty\Diff_\partial(D_4)$ factors over the automorphisms of $D^4$ as a formally smooth manifold. By the Alexander trick the latter are given by $\Omega^5 \mr{Top}(4)/\mr{O}(4)$, so we obtain that $\mr{Top}(4)/\mr{O}(4)$ is not rationally trivial. The claim then follows from the fact that $\mr{Top}/\mr{O}$ is rationally trivial \cite[Essay V]{KirbySiebenmann:FETMST}.\end{proof}

A third use for configuration space integrals lies in distinguishing embeddings. As many open problems in the topology of smooth $4$-manifolds are of this type, Theorem \ref{thm:independence} likewise rules out the direct use of configuration space integrals in their solutions. For example, using configuration space integrals to distinguish isotopy classes of embeddings of $S^3$ into $S^4$ cannot negatively resolve the $4$-dimensional smooth Schoenflies conjecture, as shown by the following result (here, the superscript $+$ indicates restriction to orientation-preserving embeddings).

\begin{proposition}The image of
	\[\mr{Emb}^{\s,+}(S^3 \times (-\epsilon,\epsilon),S^4) \lra T_\infty \mr{Emb}^{\s,+}(S^3\times (-\epsilon,\epsilon),S^4)\]
lies in a single path component. \end{proposition}
  
\begin{proof}By Theorem \ref{thm:independence}, it suffices to show that $\mr{Emb}^{\f,+}(S^3 \times (-\epsilon,\epsilon),S^4)$ is path connected. Since the topological Schoenflies conjecture holds in dimension $4$ \cite{Brown:PGST}, every locally flat embedding $S^3 \times (-\epsilon,\epsilon) \hookrightarrow S^4$ extends to an orientation-preserving locally flat embedding $\mathbb{R}^4 \hookrightarrow S^4$. This embedding can be lifted to one of formally smooth manifolds, since $\pi_4(\mr{Top}(4)/\mr{O}(4)) = 0$ \cite[Theorems 8.3B and 8.7A]{FreedmanQuinn:TO4M}. Thus, the restriction \[\mr{Emb}^{\f,+}(\mathbb{R}^4,S^4)\lra \mr{Emb}^{\f,+}(S^3 \times (-\epsilon,\epsilon),S^4)\] is surjective on path components. Finally, we have \[\Emb^{\f,+}(\mathbb{R}^4, S^4)\overset{(\ref{prop:euclidean embeddings})}{\simeq} \Emb^{\r,+}(\mathbb{R}^4, S^4)\overset{(\ref{prop:smooth and riemannian})}{\simeq} \Emb^{\s,+}(\mathbb{R}^4, S^4)\overset{(\ref{lem:smooth germ})}{\simeq} \mr{SO}(5),\] and the latter is path connected.
\end{proof}

\begin{remark} Theorem \ref{thm:independence} and the previous discussion suggests that it may be fruitful to study smooth $4$-manifolds by
\begin{enumerate}[\indent (a)]
	\item studying formally smooth 4-manifolds, and, separately,
	\item studying the difference between smooth and formally smooth 4-manifolds.
\end{enumerate}

The study of formally smooth 4-manifolds should be much like that of smooth manifolds in higher dimensions, since the Whitney trick is available under assumptions on fundamental groups \cite{FreedmanQuinn:TO4M}. In particular, it may be possible to obtain versions of the homological stability and stable homology results of Galatius and Randal-Williams in this setting (see \cite{GalatiusRandalWilliams:MSMUG} for a survey). If so, one can study the moduli space $\mathcal{M}^\f(M)$ of formally smooth manifolds isomorphic to $M$ using the methods of homotopy theory, just as one studies the moduli space $\mathcal{M}^\s(M)$ of smooth manifolds diffeomorphic to $M$ in higher dimensions. 

Next, we wish to separate the ``exotic smooth structures'' from the ``formally smooth structures'' by defining a moduli space of ``exotic'' smooth manifolds formally isomorphic to $M$. Fixing a formally smooth manifold $M$, this moduli space is defined as the homotopy fibre
\[\mathcal{M}^{\mr{ex}}(M) \coloneqq \hofibre\left[\mathcal{M}^\s(M) \to \mathcal{M}^\f(M)\right]\]
over the specified structure. As we argued above, configuration space integrals are likely blind to the topology of this moduli space.\end{remark}

\section{Embedding calculus and exotic spheres}

In this section, we prove Theorem \ref{athm:spheres}, which asserts the existence of exotic $n$-spheres $\Sigma$ for which $T_\infty\Emb^\s(\Sigma,S^n) = \varnothing$.

\subsection{Proof of Theorem \ref{athm:spheres}} Our proof uses the following convergence criterion.

\begin{proposition}\label{prop:convergence}
Let $N_1$ and $N_2$ be non-diffeomorphic closed smooth $n$-manifolds and $M$ a smooth $m$-manifold into which $N_1$ does not embed. If $m-n\geq3$ and $N_2$ embeds in $M$, then $T_\infty\Emb^\s(N_1,N_2)=\varnothing$. In particular, the map $\Emb^\s(N_1,N_2)\to T_\infty\Emb^\s(N_1,N_2)$ is a weak equivalence.
\end{proposition}

\begin{proof}
By Theorem \ref{thm:gkw} and the assumption on $N_1$, the target of the composition map \[T_\infty\Emb^\s(N_1,N_2)\times T_\infty\Emb^\s(N_2,M)\lra T_\infty\Emb^\s(N_1,M) \simeq \Emb^\s(N_1,M)\] is empty, so the source must be empty as well. The assumption on $N_2$ says that the domain of the map $\Emb^\s(N_2,M) \to T_\infty \Emb^\s(N_2,M)$ is non-empty so the right factor of the source is also non-empty. Thus the left factor is empty, as desired.
\end{proof}

The heavy lifting is handled by a collage of classical results (see also \cite[p.~408]{MahowaldThompson:EHPPH}).

\begin{theorem}[Hsiang--Levine--Szczarba, Mahowald] \label{thm:mahowald-examples}
If $n=2^j$ with $j \geq 3$, then there is an exotic $n$-sphere $\Sigma$ that does not embed in $\mathbb{R}^{n+3}$.
\end{theorem}
\begin{proof}
It suffices to show that there is an exotic $n$-sphere $\Sigma$ that embeds in $\mathbb{R}^{2n-3}$ with nontrivial normal bundle. Indeed, our assumptions on $n$ imply that $n<2(n-3)-1$, so \cite[Lemma 1.1]{HsiangLevineSzczarba:NBHSEES} then guarantees that \emph{every} embedding of $\Sigma$ in $\mathbb{R}^{2n-3}$ has nontrivial normal bundle. Since every embedding of $\Sigma$ in $\mathbb{R}^{n+3}$ has trivial normal bundle by \cite[Corollary]{Massey:NBSIES}, there can be no such embedding, or else the composite \[\Sigma \lra \mathbb{R}^{n+3} \lra \mathbb{R}^{2n-3}\] has trivial normal bundle, a contradiction.

In order to find such a $\Sigma$, it suffices by \cite[Theorem 1.2]{HsiangLevineSzczarba:NBHSEES} to find a non-zero element $\alpha\in \pi_{n-1}(\mr{SO}(n-3))$ annihilated by the maps $i \colon \pi_{n-1}(\mr{SO}(n-3))\to \pi_{n-1}(\mr{SO})\cong\mathbb{Z}$ and $J \colon \pi_{n-1}(\mr{SO}(n-3))\to \pi_{2n-4}(S^{n-3})$. 

When $n \equiv 0  \pmod 8$, $\pi_{n-1}(\mr{SO}(n-3)) \cong \mathbb{Z} \oplus \mathbb{Z}/2$ by \cite[p.~161]{Kervaire:SNHGOLG}, with the $2$-torsion generated by the image $\partial(\nu)$ of a generator $\nu\in \pi_{n}(S^{n-3})\cong\mathbb{Z}/24\mathbb{Z}$ under the connecting homomorphism
\[\partial \colon \pi_{n}(S^{n-3}) \lra \pi_{n-1} \mr{SO}(n-3)\]
of the fibration sequence $\mr{SO}(n-3) \to \mr{SO}(n-2) \to S^{n-3}$ \cite[Theorem 3(i)]{Kervaire:SNHGOLG}. 

We now prove that $\alpha = \partial(\nu)$ is in the kernel of both $i$ and $J$. According to \cite[p. 176]{HsiangLevineSzczarba:NBHSEES}, the composite \[\pi_{n}(S^{n-3})\xrightarrow{\,\partial\,} \pi_{n-1}(\mr{SO}(n-3))\xrightarrow{\,J\,}\pi_{2n-4}(S^{n-3})\] is the Whitehead product $[\iota,-]$, where $\iota\in \pi_{n-3}(S^{n-3})$ is a generator. Then $i(\alpha) \in \pi_{n-1}(\mr{SO}) \cong \mathbb{Z}$ is torsion, hence zero, while $J(\alpha)=[\iota,\nu]=0$ by \cite[page 249, (2)]{Mahowald:NIF} because $n = 2^j$ with $j \geq 3$ (\cite[Theorem 1.1.2(b)]{Mahowald:SWPS} proved there are no other cases).
\end{proof}

\begin{proof}[Proof of Theorem \ref{athm:spheres}] Set $N_1=\Sigma$ as in Theorem \ref{thm:mahowald-examples}, $N_2=S^n$, and $M=\mathbb{R}^{n+3}$ in Proposition \ref{prop:convergence}.\end{proof}

Given this result, several questions naturally arise.

\begin{question}\label{ques:exotic} Given exotic $n$-spheres $\Sigma$ and $\Sigma'$, is $T_\infty\Emb^\s(\Sigma,\Sigma')$ empty whenever $\Sigma$ and $\Sigma'$ are not diffeomorphic?\end{question}

The argument for Theorem \ref{athm:spheres} proves something stronger than the statement.

\begin{corollary}\label{cor:spheres finite stage}
For $\Sigma$ as in Theorem \ref{athm:spheres}, the map \[\Emb^\s(\Sigma, S^n) \lra T_{k}\Emb^\s(\Sigma, S^n)\] is a weak equivalence for any $k\geq n-4$.
\end{corollary}
\begin{proof}
By Theorem \ref{thm:gkw}, the map $\Emb^\s(\Sigma,\mathbb{R}^{n+3}) \to T_k \Emb^\s(\Sigma,\mathbb{R}^{n+3})$ is a $\pi_0$-surjection when $k \geq n-4$, and similarly for $S^n$ in place of $\Sigma$, so $T_k\Emb^\s(\Sigma,\mathbb{R}^{n+3})=\varnothing$ in this range, and the argument of Proposition \ref{prop:convergence} applies. 
\end{proof}

Thus the $(n-4)$th stage of the embedding calculus Taylor tower can distinguish these exotic smooth structures. On the other hand, since the first stage is given by bundle maps between tangent bundles, the fact that exotic spheres have isomorphic tangent bundles shows that the first stage does not depend on the smooth structure of $\Sigma$. Thus, in the following question, $k$ lies in the range $2 \leq k \leq n-4$.

\begin{question}What is the smallest $k$ such that $T_k \Emb^\s(\Sigma,S^n) = \varnothing$?
\end{question}

The embedding calculus Taylor tower can be modeled geometrically in terms of stratified maps of bundles over compactified configuration spaces \cite{Turchin:CFMCFMO,BoavidaWeiss:SSMCC}. Since the first stage of the tower is never empty in the case at hand, it follows that, in examples where $T_\infty\Emb^\s(\Sigma, S^n)=\varnothing$, such a stratified map exists between compactified configuration spaces of $k-1$ points that does not extend to configurations of $k$ points.

\begin{question}
Does the classification of exotic spheres admit an interpretation in terms of stratified obstruction theory applied to compactified configuration spaces?
\end{question}

\subsection{Further examples} \label{sec:examples exotic}
We indicate a few other exotic spheres for which the conclusion of Theorem \ref{athm:spheres} holds.

\begin{example}\label{example:antonelli} The paper \cite{Antonelli:SDESMR} studies the values of $n$ and $r$ for which the quotient of $\Theta_n$, the group of oriented exotic spheres under connected sum, by the subgroup of oriented exotic spheres which embed in $\mathbb{R}^{n+r}$ with trivial normal bundle is non-zero. In particular, \cite[Table 1]{Antonelli:SDESMR} provides examples of exotic $n$-spheres in dimensions $n=17,18,32,33,34,37,38$ which do not embed in $\mathbb{R}^{n+3}$.\end{example}

\begin{example}\label{example:levine} According to \cite{Levine:ACOFDK}, the generators of $\Theta_n$ for $n=8,9,10$ do not embed in $\mathbb{R}^{n+3}$.
\end{example}

In general, the homotopy theoretic problem indicated by the proof of Theorem \ref{thm:mahowald-examples}, which we believe to be of independent interest, remains open. 

\begin{question}Which elements of $\pi_{n-1}(\mr{SO}(n-3))$ lie in the common kernel of 
	\[i \colon \pi_{n-1}(\mr{SO}(n-3))\to \pi_{n-1}(\mr{SO}) \quad \text{and} \quad J \colon \pi_{n-1}(\mr{SO}(n-3))\to \pi_{2n-4}(S^{n-3})?\]\end{question}

One can also vary the target in Theorem \ref{athm:spheres}. 

\begin{example}\label{example:bp-spheres} In \cite[Theorem I]{Kervaire:HDK} it is proven that an oriented exotic $n$-sphere $\Sigma'$ embeds in $\mathbb{R}^{n+2}$ if and only if it represents an element of the subgroup $bP_{n+1} \subset \Theta_n$ of oriented exotic $n$-spheres that bound a stably parallellisable $(n+1)$-manifold. In the proof of Theorem \ref{athm:spheres}, all we used about $S^n$ is that it embeds in $\mathbb{R}^{n+3}$, so the same argument gives us that
	\[T_\infty\Emb^\s(\Sigma,\Sigma')=\varnothing\]
whenever $\Sigma'$ represents an element of $bP_{n+1}$ and $\Sigma$ is as in Examples \ref{example:antonelli} and \ref{example:levine}. (It is also true for $\Sigma$ as in Theorem \ref{thm:mahowald-examples}, but for even $n$ the group $bP_{n+1}$ is always trivial.)
\end{example}

\section{Isotopy extension for embedding calculus}

Fix manifolds $M$ and $N$ of equal dimension $n$, a compact smooth submanifold $P \subseteq N$ of codimension $0$, and an embedding $e$ of $P$ in $M$. Even though $P$ is not an object of $\MMfld^\s$ we can still define the presheaf $\Emb^\s(-,P)$, obtain a corresponding presheaf $\mathbb{E}^\s_P$ on $\DDisk_n^\s$, and define $T_\infty \Emb^\s(P,M)$ to be the derived mapping space $\Map^h_{\PSh(\DDisk_n^\s)}(\mathbb{E}^\s_P,\mathbb{E}^\s_M)$ of presheaves on $\DDisk_n^\s$. The goal of this section is to prove the following result.

\begin{theorem}\label{thm:isotopy ext} Let $M$, $N$ and $P$ be as above. If $\mr{hdim}(P) \leq \dim(M)-3$ or $P = \sqcup_I D^n$ for some finite set $I$, then the diagram
\[\begin{tikzcd} T_\infty \mr{Emb}^s_\partial(N \setminus \mathring{P},M \setminus \mathring{P}) \rar \dar & T_\infty \mr{Emb}^s(N,M) \dar \\
\ast \rar & T_\infty\mr{Emb}^s(P,M)\end{tikzcd}\] is homotopy Cartesian, where the bottom map is induced by the embedding $e$.
\end{theorem}

Removing the symbol $T_\infty$ from the statement, one obtains the conclusion of the usual isotopy extension theorem~\cite[Chapter 6]{Wall:DT}, an important tool in the study of spaces of embeddings and diffeomorphisms. Thus, Theorem \ref{thm:isotopy ext} asserts that isotopy extension holds for limits of Taylor towers. A few remarks are in order.

\begin{remark}\label{rem:isotopy-ext} \
	\begin{enumerate}[\noindent (i)]
		\item We will see that the top horizontal map is the extension-by-identity map, as in Section \ref{sec:extend-by-identity}.
		\item In this theorem, two different incarnations of embedding calculus occur; the top lefthand corner uses the version for presheaves on $\MMfld^\s_{\partial P}$, while the two righthand corners use the version for presheaves on $\MMfld^\s$.
		\item Since $P$ and $\mathring{P}$ are isotopy equivalent, the inclusion $\mathring{P} \to P$ induces a weak equivalence of presheaves $\Emb^\s(-,\mathring{P}) \to \Emb^\s(-,P)$ and thus a  weak equivalence $T_\infty\mr{Emb}^s(P,M) \to T_\infty\mr{Emb}^s(\mathring{P},M)$. Under the hypotheses of the theorem, the latter has the weak homotopy type of $\Emb^\s(\mathring{P},M)$ by Theorem \ref{thm:gkw}.
		\item A more technical hypothesis guaranteeing the conclusion of the theorem is that $\Emb^\s\big(\mathring{P} \,\sqcup\, \bigsqcup_k D^n,M\big) \to T_\infty\Emb^\s\big(\mathring{P}\, \sqcup \, \bigsqcup_k D^n,M\big)$ is a weak equivalence for all $k \geq 0$.
		\item \label{enum:neat corners} Isotopy extension for embedding calculus generalizes to spaces of neat embeddings of manifolds with corners. Here the input is as follows: $N$ and $M$ are manifolds of equal dimension $n$ with fixed embedding $\partial N \to \partial M$, and $P \subseteq N$ is a neatly embedded compact smooth submanifold of codimension 0 with corners, whose boundary $\partial P$ is the union of $\partial_0 P = \partial P \cap \partial N$ and a submanifold $\partial_1 P$, which meets at the subset of corners of $P$. Fixing a neat embedding $e \colon P \to N$ which is equal to the given embedding near $\partial_0 P$, we have the homotopy Cartesian square
		\[\begin{tikzcd} T_\infty \mr{Emb}^s_{\partial_1 P \cup \partial N \setminus \mathring{\partial_0} P}(N \setminus \mathring{P},M \setminus \mathring{P}) \rar \dar & T_\infty \mr{Emb}^s_{\partial N}(N,M) \dar \\
			\ast \rar & T_\infty\mr{Emb}^s_{\partial_0 P}(P,M).\end{tikzcd}\]
		The argument is essentially the same as that given below, but with more involved notation.
	\end{enumerate}
\end{remark}

\subsection{Proof of Theorem \ref{thm:isotopy ext}} 

\subsubsection{Complete Weiss covers} We begin with a discussion of a well-known form of locality enjoyed by embedding calculus.

\begin{definition}\label{def:weiss cover}
Let $X$ be a topological space and $1 \leq k \leq \infty$. A collection of open subsets $\mathcal{U}$ of $X$ is a \emph{Weiss $k$-cover} if every finite subset of $X$ with cardinality $\leq k$ is contained in some element of $\mathcal{U}$. A Weiss $k$-cover $\mathcal{U}$ is \emph{complete} if it contains a Weiss $k$-cover of $\bigcap_{U\in\mathcal{U}_0}U$ for every finite subset $\mathcal{U}_0\subseteq\mathcal{U}$.
\end{definition}

The following result asserts that $T_k$ has descent for complete Weiss $k$-covers. The intended application is to $k=\infty$ and $\mathbb{E}^\s_{M,\partial}$.

\begin{lemma}\label{lem:weiss descent}
Let $N$ be a smooth manifold and $1 \leq k \leq \infty$. If $F$ is a presheaf on $\MMfld_Z$, and $\mathcal{U}$ is a complete Weiss $k$-cover of $N$, each element of which contains $\partial N$, then the natural map \[T_k F(N)\lra \holim_{U\in\mathcal{U}}T_k F(U)\] is a weak equivalence.
\end{lemma}

\begin{proof}
Since derived mapping spaces convert homotopy colimits in the source to homotopy limits, it suffices to show that the natural map \[\hocolim_{U\in\mathcal{U}}\mathbb{E}^\s_{U,\partial} \lra \mathbb{E}^\s_{N,\partial}\] is a weak equivalence of presheaves on the full subcategory $\DDisk_{n,Z,\leq k}^\s$ whose objects are diffeomorphic to a disjoint union of $Z \times [0,1)$ and finitely many but at most $k$ copies of $\mathbb{R}^n$. Since homotopy colimits of presheaves are computed pointwise, it suffices to check the corresponding claim for $\Emb_\partial^\s\left(Z \times [0,1)\sqcup\bigsqcup_I\mathbb{R}^n,-\right)$ for every finite set $I$ of cardinality $\leq k$.

Assume first that $Z=\varnothing$. Given a configuration $\{p_i\}_{i\in I}\in \Conf_I(N)$ to serve as a basepoint, consider the commuting diagram \[\begin{tikzcd}
\prod_{i\in I}\Emb_{p_i}^\s(\mathbb{R}^n,N)\ar{r}\ar{d}&\prod_{i\in I} \Map_{\VVec,p_i} (T\mathbb{R}^n,TN)\ar{d}\\
\Emb^\s(\sqcup_I\mathbb{R}^n, N)\ar{d}\ar{r}&E\ar{d}\\
\Conf_I(N)\ar[equal]{r}&\Conf_I(N),
\end{tikzcd}\] where $E=\Map_{\VVec} (T\mathbb{R}^n,TN)^I|_{\Conf_I(N)}$. As in the proof of Lemma \ref{lem:riemannian pullback}, the vertical columns are fibration sequences, and the top map is a weak equivalence, so the middle map is so. The same remarks apply after replacing $N$ by $U$. The claim follows upon observing that the natural map \[\hocolim_{U\in \mathcal{U}}E|_U \lra E\] is a weak equivalence by \cite[4.6]{DuggerIsaksen:THAR}, since the collection $\{\Conf_I(U)\}_{U\in\mathcal{U}}$ is a complete cover of $\Conf_I(N)$ in the sense of \cite[4.5]{DuggerIsaksen:THAR}.

In the general case, consider the commuting diagram
\[\begin{tikzcd}\displaystyle\hocolim_{U\in\mathcal{U}}\Emb_\partial^\s\left(Z \times [0,1)\sqcup\bigsqcup_I\mathbb{R}^n,U\right) \ar{d}\ar{r} &\displaystyle\Emb_\partial^\s\left(Z \times [0,1)\sqcup\bigsqcup_I\mathbb{R}^n,N\right)\ar{d}\\
\displaystyle\hocolim_{U\in\mathcal{U}}\Emb^\s\left(\sqcup_I\mathbb{R}^n,\mathring{U}\right)\ar{r}&\Emb^\s\left(\sqcup_I\mathbb{R}^n,
\mathring{N}\right),
\end{tikzcd}\] where the vertical arrows are induced by restriction. Since the collection $\{\mathring{U}\}_{U\in\mathcal{U}}$ is a complete Weiss $k$-cover of $\mathring{N}$, the bottom arrow is a weak equivalence by the previous case. Since $\Emb^\s_\partial(Z \times [0,1), N)$ is contractible and $N$ is isotopy equivalent to its interior, isotopy extension implies that the righthand map is an equivalence, and the same considerations applied to $U$ show that the lefthand map is as well, implying the claim.
\end{proof}

\begin{remark}\label{rem:homotopy-sheaf}In fact, the map $F \to T_k F$ can be described as homotopy sheafification with respect to Weiss $k$-covers \cite[Theorem 1.2]{BoavidaWeiss:MCHS}.\end{remark}

\subsubsection{Extension-by-identity maps} \label{sec:extend-by-identity} Suppose that $M$, $N$, and $P$, are manifolds with a common boundary $Z$. Then we can form the manifolds $M \cup_{\partial} P$ and $N \cup_{\partial} P$ and construct an extension-by-identity map 
\[\Emb^s_\partial(N,M) \lra \Emb^s(N \cup_{\partial} P,M \cup_{\partial} P).\]

\begin{lemma}There is a dashed map making the following diagram commute up to preferred homotopy
	\[\begin{tikzcd} \Emb^s_\partial(N,M) \rar \dar & \Emb^s(N \cup_{\partial} P,M \cup_{\partial} P) \dar \\
	T_\infty \Emb^s_\partial(N,M) \rar[dashed] & T_\infty\Emb^s(N \cup_{\partial} P,M \cup_{\partial} P).\end{tikzcd}\]
\end{lemma}

\begin{proof}Consider the map
	\[\Emb^s_\partial(-,M) \lra \Emb^s(-\cup_{\partial} P,M \cup_{\partial P} P)\]
	of presheaves on $\DDisk^s_{n,Z}$ induced by extension-by-identity, postcomposed with
	\[\Emb^s(-\cup_{\partial} P,M \cup_{\partial} P) \lra T_\infty\Emb^s(-\cup_{\partial} P,M \cup_{\partial} P).\]
	As the target by construction satisfies descent for complete Weiss $\infty$-covers and $T_\infty \Emb_\partial^s(-,M)$ is the homotopy sheafification of $\Emb_\partial^s(-,M)$ with respect to Weiss $\infty$-covers by Remark \ref{rem:homotopy-sheaf}, this factors essentially uniquely over $T_\infty \Emb_\partial^s(-,M)$. Evaluating at $N$ we get the desired diagram.
\end{proof}

\subsubsection{Proof of Theorem \ref{thm:isotopy ext}} We proceed by applying Lemma \ref{lem:weiss descent} with $k = \infty$ and $F = \mathbb{E}_{M,\partial}^\s$ to a convenient cover. Write $\mathcal{D}_{P \subset N}$ for the collection of open subsets $U$ of $N$ that are disjoint unions of a finite number of open balls in $N \setminus P$ together with a collar neighborhood of $P$. In other words, $U$ is diffeomorphic, relative to $P$, to the manifold
\[\Big(P \cup \partial P \times [0,1)\Big) \sqcup \bigsqcup_I \mathbb{R}^n\]
for some finite set $I$. 

The reader is invited to check that $\mathcal{D}_{P \subset N}$ is a complete Weiss cover of $N$. This cover also has the following pleasant property.

\begin{lemma}\label{lem:dpm contractible}
The poset $\mathcal{D}_{P \subset N}$ is contractible.
\end{lemma}

\begin{proof}
Let $\mathcal{C}_{P\subset N}\subseteq \mathcal{D}_{P \subset N}$ denote the full subposet spanned by the objects so that the inclusion $P \hookrightarrow U$ is $0$-connected, i.e., an object of $\mathcal{C}_{P\subset N}$ is simply a collar neighborhood of $P$. A retraction and right adjoint to the inclusion of this subcategory is obtained by sending $U$ to the component of $U$ containing $P$. The claim now follows upon noting that $\mathcal{C}_{P \subset N}$ is contractible, being cofiltered.
\end{proof}

\begin{remark}
By adapting \cite[\S 5.5.2]{Lurie:HA}, something much stronger can be shown, namely that $\mathcal{D}_{P \subset N}$ is final in a sifted $\infty$-category.
\end{remark}

We now prove the isotopy extension theorem.

\begin{proof}[Proof of Theorem \ref{thm:isotopy ext}]
Suppose first that $\mr{hdim}(P) \leq \dim(M)-3$. Restricting to $U\in \mathcal{D}_{P \subset N}$ induces the commuting diagram \[\adjustbox{scale=.85,center}{\begin{tikzcd} 
T_\infty \mr{Emb}^s_\partial(N \setminus \mathring{P},M \setminus \mathring{P})\ar{r}{(1)} \dar&[-10pt]\displaystyle\holim_{U\in\mathcal{D}_{P \subset N}}T_\infty\mr{Emb}_\partial^s(U\setminus \mathring{P},M \setminus \mathring{P})\ar{d}&[-10pt] \displaystyle\holim_{U\in\mathcal{D}_{P \subset N}}\mr{Emb}_\partial^s(U\setminus \mathring{P},M \setminus \mathring{P})\ar{d}\ar{l}[swap]{(4)}\\
 T_\infty \mr{Emb}^s(N,M)\ar{r}{(2)} \dar& \displaystyle\holim_{U\in\mathcal{D}_{P \subset N}}T_\infty\mr{Emb}^s(U,M)\ar{d}& \displaystyle\holim_{U\in\mathcal{D}_{P \subset N}}\mr{Emb}^s(U,M)\ar{d}\ar{l}[swap]{(5)} \\
T_\infty\mr{Emb}^s(P,M)\ar{r}{(3)}&\displaystyle\holim_{U\in\mathcal{D}_{P \subset N}}T_\infty\mr{Emb}^s(P,M)&\displaystyle\holim_{U\in\mathcal{D}_{P \subset N}}\mr{Emb}^s(P,M)\ar{l}[swap]{(6)}.
\end{tikzcd}}\] 
where the top vertical maps are given by extension-by-identity and the bottom vertical maps by restriction to $P$. 

For each $U \in \mathcal{D}_{P \subset N}$, the rightmost column is a fibration sequence by the usual isotopy extension theorem. Since all have $e \colon P \hookrightarrow M$ as a basepoint, it remains a fibration sequence after taking homotopy limits. The claim will follow upon verifying that each of the numbered arrows is a weak equivalence. For maps (1) and (2) this follows from Lemma \ref{lem:weiss descent} applied with $k=\infty$ and $F = \smash{\mathbb{E}_{M \setminus \mathring{P},\partial}^\s}$ or $F = \smash{\mathbb{E}_{M}^\s}$ respectively, for (3) from Lemma \ref{lem:dpm contractible}, for (5) and (6) from Theorem \ref{thm:gkw} and our assumption on $P$, and for (4) from the Yoneda lemma. 

\medskip

The only modification in the case $P=\sqcup_I\mathbb{R}^n$ is for the sixth arrow, which is now an equivalence by the Yoneda lemma.
\end{proof}

\subsection{Applications of isotopy extension}

We now give some applications of Theorem \ref{thm:isotopy ext}.

\subsubsection{Rephrasing Question \ref{ques:exotic}} Let $\Sigma$ and $\Sigma'$ be exotic $n$-spheres. Fixing disks $D^n \subseteq \Sigma,\Sigma'$, we write $D_\Sigma \coloneqq \Sigma \setminus \mathring{D}^n$ for the corresponding exotic disk with boundary identified with $\partial D^n$, and similarly for $D_{\Sigma'}$.

\begin{corollary}\label{cor:disk sequence}
There is a fibration sequence sequence \[T_\infty\Emb^\s_\partial(D^n_\Sigma,D^n_{\Sigma'}) \lra T_\infty\Emb^\s(\Sigma,\Sigma') \lra \mr{O}(n+1)\]
with fibre taken over the identity.
\end{corollary}

\begin{proof}
We apply Theorem \ref{thm:isotopy ext} with $N=\Sigma$, $M=\Sigma'$, and $P=D^n$. The tangent bundle of an exotic sphere is isomorphic to that of the standard sphere (a well-known consequence of \cite[Prop.~5.4 (iv)]{BurgheleaLashof:HSD}), so $\mr{Emb}^s(D^n,\Sigma')$ is weakly equivalent to the orthogonal frame bundle of $TS^n$, which is homeomorphic to $\mr{O}(n+1)$.
\end{proof}

To connect to results about the groups $\Theta_n$, we consider a version of Question \ref{ques:exotic} for oriented exotic $n$-spheres and orientation-preserving embeddings. This question is essentially equivalent: given two oriented exotic $n$-spheres $\Sigma,\Sigma'$ then $T_\infty \Emb^{\s}(\Sigma,\Sigma')$ contains an element which reverses orientation (this is well-defined since $T_\infty$ maps to $T_1$, given by bundle maps) if and only if $T_\infty \Emb^{\s,+}(\Sigma,\smash{\overline{\Sigma}'}) \neq \varnothing$, where $\smash{\overline{\Sigma}'}$ denotes $\Sigma'$ with opposite orientation. As before, we use a superscript $+$ to denote orientation-preserving embeddings.

\begin{corollary}\label{cor:disk-vs-sphere} Let $\Sigma$ and $\Sigma'$ be oriented exotic $n$-spheres, then $T_\infty \Emb^{\s,+}(\Sigma,\Sigma')$ is non-empty if and only if $T_\infty \Emb^{\s}_\partial(D^n_\Sigma,D^n_{\Sigma'})$ is non-empty.\end{corollary}

\begin{proof}This follows directly from the oriented version of the fibration sequence in Corollary \ref{cor:disk sequence}:
\[T_\infty\Emb^\s_\partial(D^n_\Sigma,D^n_{\Sigma'}) \lra T_\infty\Emb^{\s,+}(\Sigma,\Sigma') \lra \mr{SO}(n+1).\qedhere\]	
\end{proof}


Let us define a relation on $\Theta_n$ by saying 
\[[\Sigma] \sim_{T_\infty} [\Sigma'] \quad \Longleftrightarrow \quad T_\infty \Emb^{\s,+}(\Sigma,\Sigma') \neq \varnothing.\]

\begin{lemma}This is an equivalence relation, and compatible with addition on $\Theta_n$.\end{lemma}

\begin{proof}It is easy to see it is reflexive and transitive, so remains to be proven symmetric. To do so, we claim that $T_\infty \Emb^{\s,+}(\Sigma,\Sigma') \neq \varnothing$ if and only if $T_\infty \Emb^{\s,+}(\Sigma \# \overline{\Sigma}',S^n) \neq \varnothing$; Using the previous corollary, the statement is equivalent to 
\[T_\infty \Emb^{\s}_\partial(D^n_{\Sigma},D^n_{\Sigma'}) \neq \varnothing \quad \Longleftrightarrow \quad T_\infty \Emb^{\s}_\partial(D^n_{\Sigma \# \overline{\Sigma}'},D^n) \neq \varnothing.\]
This follows from the fact that the operation of boundary connected sum with $D^n_{\overline{\Sigma}'}$, which is an instance of extension-by-the-identity, induces a map
	\[T_\infty \Emb^{\s}_\partial(D^n_{\Sigma},D^n_{\Sigma'}) \lra  T_\infty \Emb^{\s}_\partial(D^n_{\Sigma \# \overline{\Sigma}'},D^n)\]
	with homotopy inverse given by the boundary connected sum with $D^n_{\Sigma'}$. For symmetry, we use that by reversing orientations on both the domain and target, $T_\infty \Emb^{\s,+}(\Sigma \# \overline{\Sigma}',S^n) \neq \varnothing$ if and only if $T_\infty \Emb^{\s,+}(\overline{\Sigma} \# \Sigma',\overline{S^n}) \neq \varnothing$, and that $S^n$ has an orientation-reversing self-diffeomorphism.
	
	\smallskip

	To prove $\sim_{T_\infty}$ is compatible with the addition in $\Theta_n$, we argue as follows. By taking boundary connected sum with $D^n_{\Sigma''}$ or $D^n_{\overline{\Sigma}''}$ we obtain that $T_\infty \Emb^{\s}_\partial(D^n_{\Sigma},D^n_{\Sigma'}) \neq \varnothing$ if and only $T_\infty \Emb^{\s}_\partial(D^n_{\Sigma \# \Sigma''},D^n_{\Sigma' \# \Sigma''}) \neq \varnothing$, so
\[[\Sigma] \sim_{T_\infty} [\Sigma'] \quad \Longleftrightarrow \quad [\Sigma]+[\Sigma''] \sim_{T_\infty} [\Sigma']+[\Sigma''].\qedhere\]
\end{proof}

\begin{example}For $\Sigma$ as in Theorem \ref{athm:spheres} we also have that $T_\infty \Emb^\s(S^n,\Sigma) = \varnothing$.\end{example}

\begin{example}The subset $\{[\Sigma] \in \Theta_n \mid [\Sigma] \sim_{T_\infty} [S^n]\}$ is a subgroup.\end{example}

The results of \cite{BoavidaWeiss:SSMCC} shed some light on the space $T_\infty\Emb^\s_\partial(D^n_\Sigma,D^n_{\Sigma'})$. Their statement involves the operad $\mathbb{E}_n$ of little $n$-disks and its derived automorphisms.

\begin{proposition}\label{prop:tinfty disks}
There is a fibration sequence \[T_\infty\Emb^\s_\partial(D^n_\Sigma,D^n_{\Sigma'}) \lra X \lra  X'.\]
with $X$ an $\Omega^n \mr{O}(n)$-torsor and $X'$ an $\Omega^n \mr{Aut}^h(\mathbb{E}_n)$-torsor with preferred basepoint.
\end{proposition}

\begin{proof}
According to \cite[Thm.~1.1]{BoavidaWeiss:SSMCC} (with modifications for manifolds with boundary as in \cite[Section 6]{BoavidaWeiss:SSMCC}), there is a homotopy Cartesian square \[\begin{tikzcd} T_\infty\Emb^\s_\partial(D^n_\Sigma,D^n_{\Sigma'}) \rar \dar & Y \dar \\
\Map_{\VVec,\partial}(TD^n_\Sigma,TD^n_{\Sigma'}) \rar & Y', \end{tikzcd}\] where $Y$ is contractible  \cite[Thm.~1.4]{BoavidaWeiss:SSMCC} and $Y'$ is a mapping space between certain ``local configuration categories.''  We require only two pieces of information about $Y'$: (i) it is the space of compactly supported sections of a bundle over $D^n_\Sigma$, (ii) the fibres are weakly equivalent to $\mr{Aut}^h(\mathbb{E}_n)$ by \cite[Thm. 1.2]{BoavidaWeiss:SSMCC}. These facts give the identification of the righthand term, and the identification of the middle term follows from aforementioned fact about tangent bundles of exotic spheres.
\end{proof}

The action of $\mr{O}(n)$ on little $n$-disks operad by rotation gives a map $\mr{O}(n) \to \mr{Aut}^h(\mathbb{E}_n)$.  We do not know much about its effect on homotopy groups. Nevertheless, using our results on exotic spheres, we can say the following.

\begin{corollary}The map $\mr{O}(n) \to \mr{Aut}^h(\mathbb{E}_n)$ is not surjective on $\pi_n$ when $n = 2^j$ with $j \geq 3$.\end{corollary}

\begin{proof}Let $\Sigma$ be an exotic $n$-sphere as in Theorem \ref{thm:mahowald-examples}. Looping the map $\mr{O}(n) \to \mr{Aut}^h(\mathbb{E}_n)$ we obtain a map $\Omega^n \mr{O}(n) \to \Omega^n \mr{Aut}^h(\mathbb{E}_n)$ and the torsor structures on the domain and target of $X \to X'$ are compatible with this. If the map $\mr{O}(n) \to \mr{Aut}^h(\mathbb{E}_n)$ were surjective on $\pi_n$, Proposition \ref{prop:tinfty disks} would imply that $X \to X'$ is surjective on path components and hence $T_\infty\Emb^\s_\partial(D^n_\Sigma,D^n)\neq\varnothing$. Corollary \ref{cor:disk-vs-sphere} then implies a contradiction of Theorem \ref{thm:mahowald-examples}.
\end{proof}	
	
\begin{remark} The map in question \emph{is} injective on $\pi_n$, at least when $n$ is sufficiently large. Restricting to the $(n-1)$-sphere of binary operations in $\mathbb{E}_n$ and suspending produces the righthand map in $\mr{O}(n) \to \mr{Aut}^h(\mathbb{E}_n) \to \mr{Aut}^h_*(S^n)$ whose composite is the unstable $J$-homomorphism, which is injective on $\pi_n$ for $n \geq 40$ \cite{Mahowald:MHSN}.
\end{remark}

\subsubsection{Morlet's theorem for $T_\infty$} \label{sec:morlet} Setting $\Sigma=\Sigma'=S^n$, we draw the following conclusion, with $\mr{Aut}^h(\mathbb{E}_n)/\mr{O}(n)$ notation for the homotopy fibre of $B\mr{O}(n) \to B\mr{Aut}^h(\mathbb{E}_n)$.

\begin{corollary}
There are weak equivalences 
\begin{align*}\label{cor:morlet-tinfty}
&T_\infty \Diff_\partial(D^n) \simeq \Omega^{n+1}\mr{Aut}^h(\mathbb{E}_n)/\mr{O}(n),\\
&T_\infty \mr{Diff}(S^n) \simeq \mr{O}(n+1) \times \Omega^{n+1} \mr{Aut}^h(\mathbb{E}_n)/\mr{O}(n).
\end{align*}
\end{corollary}

\begin{proof}When $\Sigma = \Sigma' = S^n$, we have $D^n_\Sigma = D^n_{\Sigma'} = D^n$. In the fibration sequence \[T_\infty\Emb^\s_\partial(D^n,D^n)\lra \Omega^n \mr{O}(n)\lra  \Omega^n \mr{Aut}^h(\mathbb{E}_n)\]
	from Proposition \ref{prop:tinfty disks}, the basepoint is provided by the constant map at the identity. Thus $T_\infty\Emb^\s_\partial(D^n,D^n)$ is the fibre of a map of $n$-fold loop spaces over the unit, and hence it is group-like. This implies that $T_\infty \Diff_\partial(D^n)= T_\infty\Emb^\s_\partial(D^n,D^n)$, and the first claim follows. The second claim then follows from Corollary \ref{cor:disk sequence}, using the splitting provided by the natural action of $\mr{O}(n+1)$ on $S^n$.
\end{proof}

This result is to be compared to the classical theorem of Morlet, which asserts the same conclusion with $T_\infty$ removed and $\mr{Aut}^h(\mathbb{E}_n)$ replaced by $\mr{Top}(n)$ (see e.g.~\cite[Theorem 4.4 (b)]{BurgheleaLashof:HSD} \cite[Essay V]{KirbySiebenmann:FETMST}). Unlike Morlet's theorem, our results are valid even for $n=4$.

\begin{example}Since $\mr{Aut}(\mathbb{E}_2) \simeq \mr{O}(2)$ \cite[Thm. 8.5]{Horel:PCOGTG}, we conclude that $\mr{Diff}(S^2) \to T_\infty \mr{Diff}(S^2)$ is a weak equivalence, furnishing another example of convergence in codimension 0. In fact, embedding calculus always converges for diffeomorphisms of surfaces by \cite[Theorem A]{KrannichKupers:ECFS}.\end{example}


\subsubsection{Rephrasing the Weiss fibration sequence} Consider a manifold $M$ with $\partial M = S^{n-1}$ and disc $D^n \subset M$ such that $\partial M \cap D^n = D^{n-1} \subset \partial D^n$; that is, the disk meets the boundary of $M$ in half its boundary. Then there is a fibration sequence which---informally speaking---describes $\Diff_\partial(M)$ as built from $\Diff_\partial(D^n)$ and a certain space of self-embeddings of $M$ \cite[Section 4]{Kupers:SFRGAM} \cite[Remark 2.1.2]{Weiss:Dalian}. We will use Theorem \ref{thm:isotopy ext} to reformulate this result.

Let $T_\infty \Diff^{\cong}_\partial(M) \subseteq T_\infty \Diff_\partial(M)$ denote the union of the path components lying in the image of $\Diff_\partial(M)$. The following result asserts that, with suitable assumptions on $M$, the homotopy fibre
\[M \longmapsto \hofibre[B\Diff_\partial(M) \to BT_\infty \Diff^{\cong}_\partial(M),]
\] which we think of as the ``error term'' involved in applying embedding calculus to diffeomorphisms, is independent of $M$.

\begin{corollary}Let $M$ be a $2$-connected compact smooth manifold of dimension $n \geq 6$ with $\partial M = S^{n-1}$. The diagram
	\[\begin{tikzcd} B\Diff_\partial(D^n) \rar \dar & B\Diff_\partial(M) \dar \\
	BT_\infty \Diff^{\cong}_\partial(D^n) \rar & BT_\infty \Diff^{\cong}_\partial(M) \end{tikzcd}\]
	is homotopy Cartesian.\end{corollary}

\begin{proof} Fix an embedded closed disk $D^{n-1} \subseteq \partial M$, and let $\Emb^\s_{\partial/2}(M)$ denote the simplicial monoid of self-embeddings of $M$ fixing $D^{n-1}$ pointwise. There is the grouplike submonoid $\Emb^{\s,\cong}_{\partial/2}(M) \subseteq \Emb^\s_{\partial/2}(M)$ given by the union of the path componnents lying in the image of $\Diff_\partial(M)$. By naturality properties of embedding calculus (see \cite[Section 3, 4]{KrannichKupers:DSS} for a detailed proof of these), the diagram
	\[\begin{tikzcd} B\Diff_\partial(D^n) \rar \dar & B\Diff_\partial(M) \dar \rar & B\Emb^{\s,\cong}_{\partial/2}(M) \dar \\
	BT_\infty \Diff^{\cong}_\partial(D^n) \rar & BT_\infty \Diff^{\cong}_\partial(M) \rar & BT_\infty\Emb^{\s,\cong}_{\partial/2}(M)\end{tikzcd}\] commutes. In \cite[Lemma 3.14]{Kupers:SFRGAM}, it is verified that $M$ has handle dimension at most $n-3$ relative to $D^{n-1}$, so the righthand vertical map is a weak equivalence (strictly speaking, to apply embedding calculus as discussed above, we must remove the complement of $D^{n-1}$ in $S^{n-1}$, which gives homotopy equivalent spaces.) The top row is a fibration sequence by isotopy extension---see \cite[Remark 2.1.2]{Weiss:Dalian} and \cite[Theorem 4.17]{Kupers:SFRGAM}---and the bottom row is a fibration sequence by Theorem \ref{thm:isotopy ext} (using the extension explained in Remark \ref{rem:isotopy-ext} \eqref{enum:neat corners}).\end{proof}

\subsubsection{An example of convergence in handle codimension 2} \label{Sec:codim 2 convergence} We finish with an example of the convergence of embedding calculus Taylor tower in handle codimension 2. For the sake of readability, we omit some details regarding boundary conditions; for example, strictly speaking, to apply embedding calculus as discussed above, one must remove parts of $S^2 = \partial D^3$ not in $\partial_0 D^1_{+,\epsilon}$.

Let $D^3 \subset \mathbb{R}^3$ be the closed unit disk, which contains the interval
\[D^1 = \{(x_1,0,0) \mid x_1 \in [-1,1]\}\]
as a submanifold with boundary. We let
\[\mathbb{R}^3_+ \coloneqq \{(x_1,x_2,x_3) \mid x_1 \geq 0\}\]
denote the half plane and set $D^1_+ \coloneqq D^1 \cap \mathbb{R}^3_+.$
This is a manifold with boundary given by the union of the two points $\partial_0 D^{1}_+  = \{(0,0,0)\} = D^0$ and $\partial_1 D^{1}_+ \coloneqq \{(1,0,0)\} = D^1_+ \cap S^2$. 

The situation we will be interested in is obtained by ``thickening'' to codimension 0 the following simpler situation. By isotopy extension, there is a fibration sequence
\[\Emb^\s_{\partial}(D^1_+,D^3) \lra \Emb^\s_{\partial_1}(D^1_+,D^3) \lra \Emb^\s(D^0,D^3),\]
where the fibre is taken over the inclusion. As the middle term is contractible, we obtain the weak equivalence $\Emb^\s_{\partial}(D^1_+,D^3) \overset{\sim}\to \Omega\, \Emb^\s(D^0,D^3) \simeq \ast$,
a space-level version of the light bulb trick.

\begin{figure}
	\begin{subfigure}[t]{.45\textwidth}
	\centering
	\begin{tikzpicture}[scale=1.2]
	\draw (0,0) circle (2cm);
	\node at (0,2) [above] {$D^3$};
	\draw (-2,0) arc (180:360:2 and 0.6);
	\draw[dashed] (2,0) arc (0:180:2 and 0.6);
	\draw[thick] (0,0 ) -- node[above,fill=white]{$D^1_+$} (2,0);
	\node at (0,0)[label=left:{$\partial_0 D^1_+$}] {$\bullet$};
	\node at (2,0)[label=right:{$\partial_1 D^1_+$}] {$\bullet$};
	\end{tikzpicture}
	\subcaption{The subspaces of $D^3$ involved in the earlier part of Section \ref{Sec:codim 2 convergence}.}
	\end{subfigure}
	\hfill
	\begin{subfigure}[t]{.45\textwidth}
	\centering
	\begin{tikzpicture}[scale=1.2]
	\begin{scope}
	\clip (0,0) circle (2cm);
	\fill[Mahogany!5] (.4,.25) -- (2,.25) -- (2,-.25) -- (.4,-.25) -- cycle;
	\draw (.4,.25) -- (2,.25) -- (2,-.25) -- (.4,-.25);
	\end{scope}
	\draw (0,0) circle (2cm);
	\node at (0,2) [above] {$D^3$};
	\draw (-2,0) arc (180:360:2 and 0.6);
	\draw[dashed] (2,0) arc (0:180:2 and 0.6);
	\fill[fill=black] (0,0) circle (1pt);
	\begin{scope}[scale=.25]
	\draw[fill=Mahogany!5] (0,0) circle (2cm);
	\node at (0,2) [above,fill=white] {$D^3_\epsilon$};
	\draw (-2,0) arc (180:360:2 and 0.6);
	\draw[dashed] (2,0) arc (0:180:2 and 0.6);
	\end{scope}
	\node at (1.2,0) {$C$};
	\end{tikzpicture}
	\subcaption{The subspaces of $D^3$ involved in the latter part of Section \ref{Sec:codim 2 convergence}. The shaded region is $D^1_{+,\epsilon}$.}
	\end{subfigure}
	\caption{Several subspaces of $D^3$ which appear in Section \ref{Sec:codim 2 convergence}.}
\end{figure}
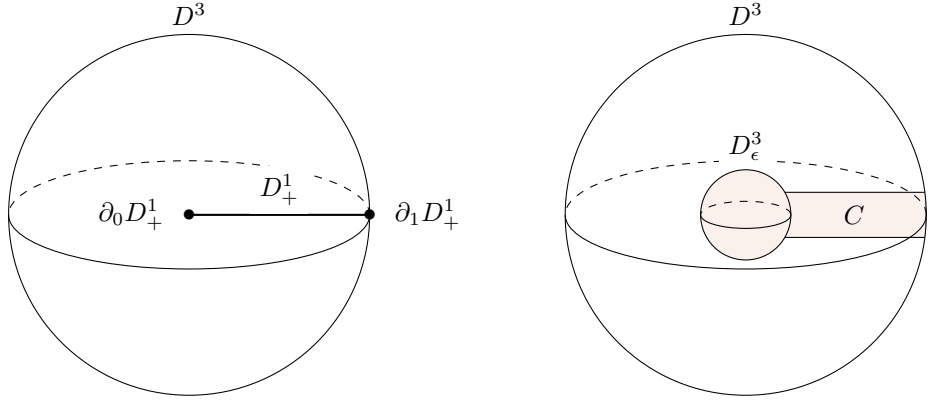

We now ``thicken'' all the submanifolds involved to codimension 0. Fixing a small $\epsilon >0$, we replace $D^0$ by $D^3_\epsilon$ and $D^1_+$ by the union of $D^3_\epsilon$ with a closed $\epsilon/2$-neighborhood of $D^1_{+,\epsilon}$ in $D^3$. We let $C$ denote the closure of $D^1_{+,\epsilon} \setminus \mathring{D}^3_{\epsilon}$ in $D^1_{+,\epsilon}$, essentially a cylinder. Its boundary intersects the larger sphere in $\partial_0 D^1_{+,\epsilon} \coloneqq D^1_{+,\epsilon} \cap S^2$ and the smaller sphere in $\partial_1 D^1_{+,\epsilon} \coloneqq D^1_{+,\epsilon} \cap S^2_\epsilon \cap \mathbb{R}^3_+$. As before, isotopy extension produces a fibration sequence with contractible middle term, whence the weak equivalence
\[\Emb^\s_{\partial_0 \cup \partial_1}(C, D^3 \setminus \mathring{D}^3_\epsilon) \overset{\sim}\lra \Omega\,\Emb^\s(D^3_\epsilon,D^3) \simeq \Omega\mr{O}(3).\]

We now show that embedding calculus captures this homotopy type; specifically, the lefthand vertical map is a weak equivalence in the commuting diagram \[\begin{tikzcd}\Emb^\s_{\partial_0 \cup \partial_1}(C, D^3 \setminus \mathring{D}^3_\epsilon) \rar \dar &[-10pt]  \Emb^\s_{\partial_1}(D^1_{+,\epsilon},D^3) \rar \dar &[-10pt] \Emb^\s(D^3_{\epsilon},D^3) \dar \\
T_\infty \Emb^\s_{\partial_0 \cup \partial_1}(C, D^3 \setminus \mathring{D}^3_\epsilon) \rar & T_\infty \Emb^\s_{\partial_1}(D^1_{+,\epsilon},D^3) \rar & T_\infty \Emb^\s(D^3_{\epsilon},D^3),\end{tikzcd}\] giving an example of convergence in codimension $2$. Since $D^3$ has handle dimension $0$, isotopy extension for embedding calculus---or rather, the extension to neat embeddings of manifolds with corners---implies that the bottom row is also a fibration sequence, so it suffices to show that the middle and righthand vertical maps are weak equivalences, both of which follow from the Yoneda lemma. For the latter map, we use that the inclusion of the interior $D^3_\epsilon$ induces a weak equivalence $T_\infty \Emb^\s(D^3_\epsilon,D^3) \simeq T_\infty \Emb^\s(\mathring{D}^3_\epsilon,D^3)$. For the former, we may similarly replace the source in $T_\infty \Emb^\s(D^1_{+,\epsilon},D^3)$ with an open collar on $\partial_1 D^1_{+,\epsilon}$.

\begin{remark}These results generalize from dimension 3 to arbitrary dimension $n \geq 3$ by changing notation: it says that embedding calculus converges in codimension 2 for embeddings of $D^{n-3} \times C$ in $D^{n-3} \times (D^3 \setminus \mathring{D}^3_\epsilon)$.
\end{remark}

\begin{appendix}

\section{Homotopy pullbacks of simplicial categories}\label{section:construction} In this appendix, we discuss a simplicial variant of a construction introduced in \cite[\S9]{Andrade:FMIEA} for topological categories.

Suppose given the following solid commuting diagram of simplicial categories \[\begin{tikzcd}
\A\times^h_{\C}\B\arrow[dashed]{rr}\arrow[dashed]{dd}&&[10pt]\B\ar{dl}[swap]{G} \arrow{dd}{P_\B}\\
&\C\arrow{rd}{P_\C}\\
\A\arrow{ur}{F}\arrow{rr}{P_\A} &&\TTop,\end{tikzcd}\] where $\TTop$ denotes the simplicial category of topological spaces. Via the structure functors to $\TTop$, objects and morphisms in $\A$, $\B$, and $\C$ have underlying spaces and maps.

\begin{construction}\label{construction:hpb}
	We define a simplicial category $\A\times^h_{\C}\B$ as follows.
	\begin{enumerate}
		\item The objects of $\A\times^h_{\C}\B$ are triples $(A,B,f)$, where $A\in\A$ and $B\in \B$ are objects with the same underlying space, and $f \colon F(A)\to G(B)$ is an isomorphism with underlying map the identity.
		\item An $n$-simplex in the mapping space from $(A_1,B_1,f_1)$ to $(A_2,B_2,f_2)$ is a triple $(\varphi, \psi, \gamma)$, where $\varphi\in \Map_\A(A_1, A_2)_n$ and $\psi\in\Map_\B(B_1, B_2)_n$ have the same underlying simplex in $\TTop$, and $\gamma$ is a path $f_2\circ F(\varphi)\implies G(\psi)\circ f_1$ in $(\Map_\C(F(A_1),G(B_2))^{\Delta^1})_n$ covering the constant path.
		\item Composition is induced by composition in $\A$, $\B$, and $\C$, and the diagonal of $\Delta^1$.
	\end{enumerate}
\end{construction}

The notation $\A\times^h_{\C}\B$ is justified by the following result. whose proof we defer to the end of this subsection and may be skipped on a first reading.

\begin{proposition}\label{prop:homotopy pullback}
Suppose that
\begin{enumerate} 
\item each of the simplicial sets $\Map_\B(B_1,B_2)$ and $\Map_\C(F(A_1), G(B_2))$ is a Kan complex, and 
\item each of the structure maps $\Map_\B(B_1,B_2)\to \Map_{\TTop}(P_\B(B_1), P_\B(B_2))$ and $\Map_\C(F(A_1),G(B_2))\to \Map_{\TTop}(P_\A(A_1), P_\B(B_2))$ is a Kan fibration.
\end{enumerate}
The diagram
\[
\begin{tikzcd}
\Map_{\A\times^h_{\C}\B}\left((A_1,B_1,f_1),(A_2,B_2,f_2)\right)\ar{r}\ar{d}&\Map_\B(B_1,B_2,)\ar{d}\\
\Map_\A(A_1,A_2)\ar{r}&\Map_{\C}(F(A_1), G(B_2))
\end{tikzcd}
\] is homotopy Cartesian. 
\end{proposition} 

Note that the diagram in question commutes only up to specified homotopy. 

\begin{remark}
Proposition \ref{prop:homotopy pullback} implies that $\A\times_\C^h\B$ is often the homotopy pullback of $\A$ and $\B$ over $\C$ in the Bergner model structure on simplicial categories \cite{Bergner:MCSCSC}; specifically, we require the assumptions of the proposition to hold for all objects, and we require that $\mr{Ho}(P_\B)$ and $\mr{Ho}(P_\C)$ be isofibrations. Therefore, we think of $\A\times^h_\C\B$ as a (particularly convenient) model for the pullback of $\infty$-categories, whose homotopy theory is captured by the Bergner model structure.
\end{remark}

\begin{construction}\label{construction:functor}
	Suppose given $\A$, $\B$, and $\C$ as above. Let $\D$ be a simplicial category equipped with simplicial functors $H \colon \D\to \A$ and $K \colon \D\to \B$ over $\TTop$, together with the natural isomorphism $\chi$ in the diagram \[\begin{tikzcd} \D\arrow{dd}[swap]{H} \arrow{rr}{K} &[-5pt]&[-5pt]\B\arrow{dd}{G}\\[-10pt]\\
	\A\arrow[Rightarrow]{uurr}{\chi} \arrow{rr}[swap]{F} && \C,\end{tikzcd} \] we obtain a functor $\D\to \A_2\times^h_{\C_2}\B_2$ as follows.
	\begin{enumerate}
		\item The object $D\in \D$ is sent to the triple $(H(D), K(D), \chi_D)$.
		\item The $n$-simplex $\sigma\in \Map_\D(D_1,D_2)$ is sent to the triple consisting of $H(\sigma)$, $K(\sigma)$, and the constant path at $\chi_{D_2}\circ H(\sigma)=K(\sigma)\circ \chi_{D_1}$.
	\end{enumerate}
\end{construction}

To prove Proposition \ref{prop:homotopy pullback}, it will be convenient to put ourselves in a more general setting. Suppose given the following commutative diagram of simplicial sets \[\begin{tikzcd} 
&X\ar{d}[swap]{g}\ar[bend left=35]{ddr}{p_X}\\
Z\ar[bend right=35]{drr}{p_Z}\ar{r}{h}&Y\ar{dr}{p_Y}\\
&&W.
\end{tikzcd}\] Write $P$ for the standard model of the homotopy pullback of $X$ and $Z$ over $Y$; explicitly, $P$ is the limit of the diagram \[\begin{tikzcd}
X\ar{dr}{g}&&Y^{\Delta^1}\ar{dr}{\mr{ev_1}}\ar{dl}[swap]{\mr{ev}_0}&&Z\ar{dl}[swap]{h}\\[-5pt]
&Y&&Y.
\end{tikzcd}\] Finally, write $P_0$ for the pullback in the diagram \[\begin{tikzcd}
P_0\ar{d}\ar{r}{\iota}&P\ar{d}{q}\\
W\ar{r}&W^{\Delta^1},
\end{tikzcd}\] where the bottom arrow is the inclusion of the constant maps and $q$ is the composition of the projection to $Y^{\Delta^1}$ with $(p_Y)^{\Delta^1}$. We think of $P_0$ as the subspace of the homotopy pullback lying over constant paths in $W$. In the example of interest, $X$, $Y$, and $Z$ are mapping spaces in the relevant simplicial categories, and $W$ is the corresponding mapping space in $\TTop$.

The topological analogue of the following result is asserted in \cite[\S 9]{Andrade:FMIEA} We include a proof for the sake of completeness.

\begin{lemma}\label{lem:constant path model}
	If $p_Y$ and $p_Z$ are fibrations, then $\iota$ is a weak equivalence.
\end{lemma}
\begin{proof}
	Given the solid commuting diagram \[\begin{tikzcd}
	\partial\Delta^n\ar{d}\ar{r}&P_0\ar{d}{\iota}\\
	\Delta^n\ar{r}\ar[dashed]{ur}&P,
	\end{tikzcd}\] we will produce the dashed arrow making the top triangle commute and the bottom triangle commute up to homotopy fixing $\partial \Delta^n$. First, using the assumption that $p_Z$ is a fibration, we solve the lifting problem  \[\begin{tikzcd}
	\displaystyle\Delta^n\times\Delta^0\bigsqcup_{\partial \Delta^n\times\Delta^0}\partial\Delta^n\times \Delta^1\ar{d}\ar{r}&P\ar{r}&Z\ar{d}{p_Z}\\[-5pt]
	\Delta^n\times\Delta^1\ar{rr}\ar[dashed]{urr}&&W,
	\end{tikzcd}\] where the bottom map is the adjunct of the composite $\Delta^n\to P\to Y^{\Delta^1}\to W^{\Delta^1}$, and the lefthand map is induced by the inclusion of the vertex $0$. Composing with $h$ and passing back through the adjunction, we obtain the top map in the commuting diagram  
	\[\begin{tikzcd}
	\Delta^n\ar{d}\ar{rr}&&Y^{\Delta^1}\ar{d}{\mr{ev}_0}\\
	P\ar{r}&Y^{\Delta^1}\ar{r}{\mr{ev}_1}&Y.
	\end{tikzcd}\] There is an induced map $\Delta^n\times\Lambda^2_1\to Y$, and we use the assumption that $p_Y$ is a fibration to solve the lifting problem
	\[\begin{tikzcd}
	\displaystyle\Delta^n\times\Lambda^2_1\bigsqcup_{\partial \Delta^n\times\Lambda^2_1}\partial\Delta^n\times \Delta^2\ar{d}\ar{rr}&&Y\ar{d}{p_Y}\\[-5pt]
	\Delta^n\times\Delta^2\ar{r}\ar[dashed]{urr}&\Delta^n\times\Delta^1\ar{r}&W.
	\end{tikzcd}\] Restricting to the third face of $\Delta^2$, we obtain by adjunction the middle map in the commuting diagram 
	\[\begin{tikzcd}
	&&\Delta^n\ar{ddll}\ar{dd}\ar{ddrr}\\ \\
	X\ar{dr}{g}&&Y^{\Delta^1}\ar{dr}{\mr{ev_1}}\ar{dl}[swap]{\mr{ev}_0}&&Z\ar{dl}[swap]{h}\\[-5pt]
	&Y&&Y,
	\end{tikzcd}\] where the lefthand map is the composite $\Delta^n\to P\to X$, and the righthand map is the restriction of our earlier lift $\Delta^n\times\Delta^1\to Z$ to the vertex $1$. The resulting map $\Delta^n\to P$ factors through $P_0$ and restricts to the original map on $\partial \Delta^n$ by construction. Also by construction, the righthand square of the above diagram comes equipped with a homotopy relative to $\partial\Delta^n$, which furnishes the desired homotopy.
\end{proof}

\begin{proof}[Proof of Proposition \ref{prop:homotopy pullback}]
	The first assumption guarantees that the standard model for the homotopy pullback has the correct weak equivalence type. The second assumption permits the invocation of Lemma \ref{lem:constant path model}, which guarantees that the canonical map from $\Map_{\A\times^h_{\C}\B}\left((A_1,B_1,f_1),(A_2,B_2,f_2)\right)$ to the standard model for the homotopy pullback is a weak equivalence.
\end{proof}
\end{appendix}

\bibliographystyle{amsalpha}
\bibliography{references}

\providecommand{\bysame}{\leavevmode\hbox to3em{\hrulefill}\thinspace}
\providecommand{\MR}{\relax\ifhmode\unskip\space\fi MR }
\providecommand{\MRhref}[2]{%
  \href{http://www.ams.org/mathscinet-getitem?mr=#1}{#2}
}
\providecommand{\href}[2]{#2}
\begin{thebibliography}{BdBW18}

\bibitem[And12]{Andrade:FMIEA}
R.~Andrade, \emph{From manifolds to invariants of {$E_n$}-algebras}, 2012,
  arXiv:1210.7909.

\bibitem[Ant71]{Antonelli:SDESMR}
P.~L. Antonelli, \emph{On stable diffeomorphism of exotic spheres in the
  metastable range}, Canadian J. Math. \textbf{23} (1971), 579--587.
  \MR{283810}

\bibitem[AS22]{AroneSzymik:SKCESS}
G.~Arone and M.~Szymik, \emph{Spaces of knotted circles and exotic smooth
  structures}, Canad. J. Math. \textbf{74} (2022), no.~1, 1--23. \MR{4379395}

\bibitem[BdBW13]{BoavidaWeiss:MCHS}
P.~Boavida~de Brito and M.~Weiss, \emph{Manifold calculus and homotopy
  sheaves}, Homology Homotopy Appl. \textbf{15} (2013), no.~2, 361--383.
  \MR{3138384}

\bibitem[BdBW18]{BoavidaWeiss:SSMCC}
\bysame, \emph{Spaces of smooth embeddings and configuration categories}, J.
  Topol. \textbf{11} (2018), no.~1, 65--143. \MR{3784227}

\bibitem[Ber07]{Bergner:MCSCSC}
J.~Bergner, \emph{A model category structure on the category of simplicial
  categories}, Transactions of the American Mathematical Society \textbf{359}
  (2007), 2043--2058.

\bibitem[BL74]{BurgheleaLashof:HSD}
D.~Burghelea and R.~Lashof, \emph{The homotopy type of the space of
  diffeomorphisms. {I}, {II}}, Trans. Amer. Math. Soc. \textbf{196} (1974),
  1--36; ibid. 196 (1974), 37--50. \MR{356103}

\bibitem[Bro60]{Brown:PGST}
M.~Brown, \emph{A proof of the generalized {S}choenflies theorem}, Bull. Amer.
  Math. Soc. \textbf{66} (1960), 74--76. \MR{117695}

\bibitem[DI04]{DuggerIsaksen:THAR}
D.~Dugger and D.~Isaksen, \emph{Topological hypercovers and 1-realizations},
  Mathematische Zeitschrift \textbf{246} (2004), 667--689.

\bibitem[EK71]{EdwardsKirby:DSI}
R.~D. Edwards and R.~C. Kirby, \emph{Deformations of spaces of imbeddings},
  Ann. of Math. (2) \textbf{93} (1971), 63--88. \MR{283802}

\bibitem[FQ90]{FreedmanQuinn:TO4M}
M.~H. Freedman and F.~Quinn, \emph{Topology of 4-manifolds}, Princeton
  Mathematical Series, vol.~39, Princeton University Press, Princeton, NJ,
  1990. \MR{1201584}

\bibitem[GK15]{GoodwillieKlein:MDSE}
T.~G. Goodwillie and J.~R. Klein, \emph{Multiple disjunction for spaces of
  smooth embeddings}, J. Topol. \textbf{8} (2015), no.~3, 651--674.
  \MR{3394312}

\bibitem[GKW03]{GoodwillieKleinWeiss:HSTDEC}
T.~G. Goodwillie, J.~R. Klein, and M.~S. Weiss, \emph{A {H}aefliger style
  description of the embedding calculus tower}, Topology \textbf{42} (2003),
  no.~3, 509--524. \MR{1953238}

\bibitem[GRW20]{GalatiusRandalWilliams:MSMUG}
S.~Galatius and O.~Randal-Williams, \emph{{Moduli spaces of manifolds: a user's
  guide}}, in Handbook of Homotopy Theory, H.~Miller (Ed.), New York: Chapman
  and Hall/CRC, 2020, pp.~443--485.

\bibitem[GW99]{GoodwillieWeiss:EPVITII}
T.~G. Goodwillie and M.~Weiss, \emph{Embeddings from the point of view of
  immersion theory. {II}}, Geom. Topol. \textbf{3} (1999), 103--118.
  \MR{1694808}

\bibitem[HLS65]{HsiangLevineSzczarba:NBHSEES}
W.-c. Hsiang, J.~Levine, and R.~H. Szczarba, \emph{On the normal bundle of a
  homotopy sphere embedded in {E}uclidean space}, Topology \textbf{3} (1965),
  173--181. \MR{175138}

\bibitem[Hor17]{Horel:PCOGTG}
G.~Horel, \emph{Profinite completion of operads and the
  {G}rothendieck-{T}eichm\"{u}ller group}, Adv. Math. \textbf{321} (2017),
  326--390. \MR{3715714}

\bibitem[Hus94]{Husemoller:FB}
D.~Husemoller, \emph{Fibre bundles}, third ed., Graduate Texts in Mathematics,
  vol.~20, Springer-Verlag, New York, 1994. \MR{1249482}

\bibitem[Ker60]{Kervaire:SNHGOLG}
M.~A. Kervaire, \emph{Some nonstable homotopy groups of {L}ie groups}, Illinois
  J. Math. \textbf{4} (1960), 161--169. \MR{0113237}

\bibitem[Ker65]{Kervaire:HDK}
\bysame, \emph{On higher dimensional knots}, Differential and {C}ombinatorial
  {T}opology ({A} {S}ymposium in {H}onor of {M}arston {M}orse), Princeton Univ.
  Press, Princeton, N.J., 1965, pp.~105--119. \MR{0178475}

\bibitem[Kis64]{Kister:MBFB}
J.~M. Kister, \emph{Microbundles are fibre bundles}, Ann. of Math. (2)
  \textbf{80} (1964), 190--199. \MR{180986}

\bibitem[KK21]{KrannichKupers:ECFS}
M.~Krannich and A.~Kupers, \emph{Embedding calculus for surfaces},
  arXiv:2101.07885, 2021.

\bibitem[KK22]{KrannichKupers:DSS}
\bysame, \emph{The {D}isc-structure space}, 2022, arXiv:2205.01755.

\bibitem[Kon94]{Kontsevich:FDLT}
M.~Kontsevich, \emph{Feynman diagrams and low-dimensional topology}, First
  {E}uropean {C}ongress of {M}athematics, {V}ol. {II} ({P}aris, 1992), Progr.
  Math., vol. 120, Birkh\"{a}user, Basel, 1994, pp.~97--121. \MR{1341841}

\bibitem[K{\"o}r17]{Korschgen:DKEIETEP}
A.~K{\"o}rschgen, \emph{{D}wyer-{K}an equivalences induce equivalences on
  topologically enriched presheaves}, arXiv:1704.07472, 2017.

\bibitem[KRW20]{KupersRandalWilliams:CTSA}
A.~Kupers and O.~Randal-Williams, \emph{The cohomology of {T}orelli groups is
  algebraic}, Forum Math. Sigma \textbf{8} (2020), Paper No. e64, 52.
  \MR{4190064}

\bibitem[KS77]{KirbySiebenmann:FETMST}
R.~C. Kirby and L.~C. Siebenmann, \emph{Foundational essays on topological
  manifolds, smoothings, and triangulations}, Princeton University Press,
  Princeton, N.J.; University of Tokyo Press, Tokyo, 1977, With notes by John
  Milnor and Michael Atiyah, Annals of Mathematics Studies, No. 88.
  \MR{0645390}

\bibitem[Kup19]{Kupers:SFRGAM}
A.~Kupers, \emph{Some finiteness results for groups of automorphisms of
  manifolds}, Geom. Topol. \textbf{23} (2019), 2277--2333 (electronic).

\bibitem[Lev65]{Levine:ACOFDK}
J.~Levine, \emph{A classification of differentiable knots}, Ann. of Math. (2)
  \textbf{82} (1965), 15--50. \MR{180981}

\bibitem[Lur17]{Lurie:HA}
J.~Lurie, \emph{Higher algebra}, 2017,
  \url{http://people.math.harvard.edu/~lurie/papers/HA.pdf}.

\bibitem[Mah65]{Mahowald:SWPS}
M.~Mahowald, \emph{Some {W}hitehead products in {$S\sp{n}$}}, Topology
  \textbf{4} (1965), 17--26. \MR{178467}

\bibitem[Mah67]{Mahowald:MHSN}
\bysame, \emph{The metastable homotopy of {$S\sp{n}$}}, Memoirs of the American
  Mathematical Society, No. 72, American Mathematical Society, Providence,
  R.I., 1967. \MR{0236923}

\bibitem[Mah77]{Mahowald:NIF}
\bysame, \emph{A new infinite family in {${}\sb{2}\pi_{*}{}^s$}}, Topology
  \textbf{16} (1977), no.~3, 249--256. \MR{445498}

\bibitem[Mas59]{Massey:NBSIES}
W.~S. Massey, \emph{On the normal bundle of a sphere imbedded in {E}uclidean
  space}, Proc. Amer. Math. Soc. \textbf{10} (1959), 959--964. \MR{109351}

\bibitem[Mil64]{Milnor:MBI}
J.~Milnor, \emph{Microbundles. {I}}, Topology \textbf{3} (1964), no.~suppl,
  suppl. 1, 53--80. \MR{161346}

\bibitem[MT95]{MahowaldThompson:EHPPH}
M.~Mahowald and R.~D. Thompson, \emph{The {$EHP$} sequence and periodic
  homotopy}, Handbook of algebraic topology, North-Holland, Amsterdam, 1995,
  pp.~397--423. \MR{1361895}

\bibitem[Pri20]{Prigge:TCFBSEC}
N.~Prigge, \emph{On tautological classes of fibre bundles and self-embedding
  calculus}, 2020.

\bibitem[Ran96]{Ranicki:OTH}
A.~A. Ranicki, \emph{On the {H}auptvermutung}, The {H}auptvermutung book,
  $K$-Monogr. Math., vol.~1, Kluwer Acad. Publ., Dordrecht, 1996, pp.~3--31.
  \MR{1434101}

\bibitem[Sie72]{Siebenmann:DOHOSS}
L.~C. Siebenmann, \emph{Deformation of homeomorphisms on stratified sets. {I},
  {II}}, Comment. Math. Helv. \textbf{47} (1972), 123--136; ibid. 47 (1972),
  137--163. \MR{319207}

\bibitem[Tur13]{Turchin:CFMCFMO}
V.~Turchin, \emph{Context-free manifold calculus and the
  {F}ulton-{M}ac{P}herson operad}, Algebr. Geom. Topol. \textbf{13} (2013),
  no.~3, 1243--1271.

\bibitem[Vir15]{Viro:SS1}
O.~Viro, \emph{Space of smooth 1-knots in a 4-manifold: is its algebraic
  topology sensitive to smooth structures?}, Arnold Math. J. \textbf{1} (2015),
  no.~1, 83--89. \MR{3331970}

\bibitem[Vol06]{Volic:FTKICF}
I.~Voli\'{c}, \emph{Finite type knot invariants and the calculus of functors},
  Compos. Math. \textbf{142} (2006), no.~1, 222--250. \MR{2197410}

\bibitem[Wal16]{Wall:DT}
C.~T.~C. Wall, \emph{Differential topology}, Cambridge Studies in Advanced
  Mathematics, vol. 156, Cambridge University Press, Cambridge, 2016.
  \MR{3558600}

\bibitem[Wat18]{Watanabe:ENERHGD}
T.~Watanabe, \emph{Some exotic nontrivial elements of the rational homotopy
  groups of $\mathrm{Diff}({S}^4)$}, 2018, arXiv:1812.02448.

\bibitem[Wei99]{Weiss:EPVITI}
M.~Weiss, \emph{Embeddings from the point of view of immersion theory: Part
  {I}}, Geom. Topol. \textbf{3} (1999), no.~1, 67--101.

\bibitem[Wei11]{Weiss:EPVITIerratum}
\bysame, \emph{Erratum to the article {E}mbeddings from the point of view of
  immersion theory: {P}art {I}}, Geom. Topol. \textbf{15} (2011), no.~1,
  407--409. \MR{2776849}

\bibitem[Wei21]{Weiss:Dalian}
M.~S. Weiss, \emph{Rational {P}ontryagin classes of {E}uclidean fiber bundles},
  Geom. Topol. \textbf{25} (2021), no.~7, 3351--3424. \MR{4372633}

\end{thebibliography}

\vspace{1cm}

\end{document}